\documentclass[11pt]{article}
\usepackage{geometry}                
\usepackage{graphicx}
\usepackage{amssymb}
\usepackage{amsmath}
\usepackage{amsthm}
\usepackage{epstopdf}
\usepackage{color}
\usepackage[colorinlistoftodos,prependcaption,textsize=tiny]{todonotes}
\usepackage{todonotes}
\theoremstyle{plain}
\newtheorem{theorem}{Theorem}
\newtheorem{lemma}[theorem]{Lemma}
\newtheorem{corollary}[theorem]{Corollary}
\newtheorem{proposition}[theorem]{Proposition}
\newtheorem{assumption}[theorem]{Assumption}

\theoremstyle{plain}
\newtheorem{definition}[theorem]{Definition}
\newtheorem{example}[theorem]{Example}

\newtheorem{remark}[theorem]{Remark}

\DeclareMathOperator{\Hess}{Hess}

\newcommand{\cA}{\mathcal{A}}

\newcommand{\cC}{\mathcal{C}}
\newcommand{\cD}{\mathcal{D}}

\newcommand{\cG}{\mathcal{G}}

\newcommand{\cL}{\mathcal{L}}

\newcommand{\cN}{\mathcal{N}}
\newcommand{\cO}{\mathcal{O}}

\newcommand{\cV}{\mathcal{V}}

\newcommand{\E}{\mathbb{E}}

\newcommand{\R}{\mathbb{R}}

\title{Affine invariant interacting Langevin dynamics for Bayesian inference}
\author{Alfredo Garbuno-Inigo\thanks{Computing and Mathematical Sciences, California Institute of Technology,
1200 East California Boulevard, 91125 Pasadena, United States,
{\tt agarbuno@caltech.edu}} \and Nikolas N\"usken\thanks{Institute of Mathematics, University of Potsdam,
Karl-Liebknecht-Str. 24/25, D-14476 Potsdam, Germany,
{\tt nuesken@uni-potsdam.de}} \and Sebastian Reich\thanks{
Institute of Mathematics, University of Potsdam, Karl-Liebknecht-Str. 24/25, D-14476 Potsdam, Germany, and
Department of Mathematics and Statistics, University of Reading,
Reading RG6 6AX, England, {\tt sereich@uni-potsdam.de}
}
}

\begin{document}
\maketitle

\begin{abstract} We propose a computational method (with acronym ALDI) for sampling from a given target distribution based on first-order
(overdamped) Langevin dynamics which satisfies the property of affine invariance. The central idea of ALDI is to run
an ensemble of particles with their empirical covariance serving as a preconditioner for their underlying Langevin
dynamics. ALDI does not require taking the inverse or square root of the empirical covariance matrix, which
enables application to high-dimensional sampling problems. The theoretical properties of ALDI are studied in terms of non-degeneracy and
ergodicity. Furthermore, we study its connections to diffusion on Riemannian manifolds and Wasserstein gradient flows.

Bayesian inference serves as a main application area for ALDI. In case of a forward problem with additive Gaussian measurement errors,
ALDI allows for a gradient-free approximation in the spirit of the ensemble Kalman filter. A computational comparison between gradient-free
and gradient-based ALDI is provided for a PDE constrained Bayesian inverse problem.
\end{abstract}

\noindent {\bf AMS(MOS) subject classifications:} 65N21, 62F15, 65N75, 65C30, 90C56

\section{Introduction}

In this paper, we propose an efficient sampling method for Bayesian inference which is based
on first-order (overdamped) Langevin dynamics \cite{sr:P14} and which satisfies the property of affine invariance
\cite{sr:GW10}. Here affine invariance of a computational method refers to the fact that a method is invariant
under an affine change of coordinates. A classical example is provided by Newton's method, while standard
gradient descent is not affine invariant. The importance of affine invariance as a general guiding principle for
the design of Monte Carlo sampling methods was first highlighted in the pioneering contribution \cite{sr:GW10}.

Langevin dynamics based sampling methods, on the other hand, have a long history in statistical physics \cite{sr:RDF78} and
computational statistics \cite{sr:RS03}. An important step towards affine
invariant Langevin sampling methods was taken through the introduction of Riemannian manifold Langevin Monte Carlo methods
in \cite{sr:GC11} with the metric tensor given by the Fisher information matrix. However, the Fisher information matrix
is typically not available in closed form and/or is difficult to approximate numerically. Instead, an alternative approach was
put forward in the unpublished Master thesis \cite{sr:Greengard2015}, where an ensemble of Langevin samplers is combined to provide
an empirical covariance matrix resulting in a preconditioned affine invariant MALA algorithm
(see Section \ref{sec:history} for more details). This methodology was put into the wider context of dynamics-based sampling methods
in \cite{sr:LMW18} with a focus on second-order Langevin dynamics.

An interesting link between ensembles of Langevin samplers and the ensemble Kalman filter
\cite{sr:evensen,sr:stuart15,sr:reichcotter15}, both relying on ensemble based empirical covariance matrices,
has been established more recently in \cite{sr:GIHLS19} leading to a nonlinear
Fokker--Planck equation for the associated mean-field equations and an associated Kalman--Wasserstein gradient flow structure in the space of probability measures. The same gradient flow structure has been previously identified for the time-continuous ensemble Kalman--Bucy filter mean-field equations \cite{sr:cotterreich,sr:reichcotter15}. Furthermore, if applied to a Bayesian inverse problem with additive Gaussian measurement errors and nonlinear forward map, a gradient-free approximate Langevin dynamics formulation has been proposed \cite{sr:GIHLS19} which is again based on ideas previously exploited in the ensemble Kalman filter literature \cite{sr:evensen,sr:br11}.

The present paper builds upon the unpublished note \cite{sr:NR19}, which identifies a statistically consistent finite ensemble
size implementation of the mean-field equations put forward in \cite{sr:GIHLS19}.  More precisely, the proposed
interacting Langevin dynamics possesses the desired posterior target measure as an invariant measure provided an appropriate
correction term is added, which is due to the multiplicative noise in the preconditioned Langevin system. The correction term vanishes in
the mean-field limit. Furthermore, the invariance of our finite ensemble size evolution equations (with acronym ALDI\footnote{The acronym stands for a permutation of the capital letters in {A}ffine {I}nvariant {L}angevin {D}ynamics.}) under affine coordinate transformations is established through a particular choice of the multiplicative noise term, amongst all choices consistent with the desired underlying Fokker--Planck equation. We emphasise that ALDI is straightforward to implement, does not require inversion or other matrix factorisations of the empirical covariance matrices (which is important for high-dimensional problems) and is applicable to a wide range of sampling problems.

We have already emphasised  that related computational methods have been considered in the literature before.
However, none of these contributions has investigated the non-degeneracy and ergodicity properties of such methods.
Hence, proof of non-degeneracy and ergodicity of ALDI provides a key theoretical contribution of our paper which holds provided the ensemble size, $N$, and the dimension, $D$, of the underlying random variable satisfy $N>D+1$ and the empirical covariance matrix is non-degenerate at initial time.

Finally, a gradient-free formulation of ALDI in the spirit of  \cite{sr:GIHLS19} is proposed for Bayesian inverse problems with additive Gaussian measurement errors. While the invariance of the posterior distribution is lost when making the gradient-free approximation, except for Gaussian likelihood functions, affine invariance is maintained. Numerical experiments are conducted for a PDE constrained Bayesian inference problem.
The numerical results indicate in particular that it is entirely sufficient to implement ALDI with $N = D+2$ particles;
the minimum size required for ergodicity to hold. Thus the gradient-free implementation indeed provides an accurate and computationally
inexpensive alternative.

The remainder of this paper is structured as follows. The subsequent Section \ref{sec:math prob} establishes the mathematical setting
of the sampling problems considered in this paper and provides a unifying mathematical framework for ensemble-based
first-order Langevin dynamics. Given this framework, we formulate the key algorithmic requirements on the ensemble formulation
proposed in this paper. We introduce the concept of affine invariance and prove affine invariance for the nonlinear Fokker--Planck equations
put forward in \cite{sr:GIHLS19}. The algorithmic contributions of this paper can be found in Section \ref{sec:Algorithms}. More
specifically, the novel ALDI method is put forward in Section \ref{sec:ALDI} and its gradient-free variant in
Section \ref{sec:gradient_free}. Both methods are put into the context of previous algorithmic work in Section \ref{sec:history}. Our theoretical investigations are summarised in Section \ref{sec:properties}, where the affine invariance, non-degeneracy and ergodicity
of ALDI are proven. We also put our approach into the perspective of diffusion processes on Riemannian manifolds \cite{sr:GC11,livingstone2014information} and Wasserstein gradient flows \cite{ambrosio2008gradient,villani2008optimal}.
The importance of the correction term is demonstrated for
a PDE constrained inverse problem \cite{sr:GIHLS19} in the numerical example Section \ref{sec:numerics}.
We also compare the performance of the gradient-based and gradient-free formulations of  ALDI and find that both lead to comparable numerical results with the gradient-free formulation however much cheaper to implement. We conclude the paper with a summary section.


%

\section{Mathematical problem formulation}
\label{sec:math prob}

We consider the computational problem of producing samples from a random variable $u$ with values in $\R^D$ and given probability density function (PDF)
\begin{equation}
\pi_\ast(u) = \frac{1}{Z} \exp(-\Phi(u)),
\end{equation}
where $\Phi:\R^D \to \R$ is an appropriate potential and
\begin{equation}
Z:=\int_{\R^{D}}\exp(-\Phi(u))\,{\rm d}u < \infty
\end{equation}
a normalisation constant.

\begin{example}[Bayesian inverse problems] \label{ex:BIP}
The computational Bayesian inverse problem (BIP) of sampling a random variable $u$ conditioned on an observation
$y_{\rm obs} \in\R^{K}$ with forward model
\begin{equation}\label{eq:IP}
y = \cG(u)+\xi,
\end{equation}
serves as the main motivation of this paper. Here, $\cG:\R^{D}\to\R^{K}$ denotes some nonlinear forward map and the mean zero
$\mathbb{R}^K$-valued Gaussian random variable $\xi$
represents measurement errors with positive definite error covariance matrix $R \in \mathbb{R}^{K\times K}$.
We assume that $\xi$  and $u \sim \pi_0$ are independent. Then, by Bayes' theorem, the distribution of the
conditional random variable $u| y_{\rm obs}$ is determined by
\begin{equation} \label{eq:posterior}
\pi({\rm d}u|y_{\rm obs})=\frac{1}{Z}\exp(-l(u ; y_{\rm obs}))\,\pi_0({\rm d}u),
\end{equation}
with the least-squares misfit function\footnote{Here we have introduced the weighted $l_2$-norm $\|a\|_B = (a^{\rm T} B^{-1} a)^{1/2}$ for
any symmetric positive-definite matrix $B$.}
\begin{equation} \label{eq:misfit}
l(u ; y_{\rm obs})=\frac{1}{2}\lVert R^{-\frac{1}{2}}(y_{\rm obs}-\cG(u)) \rVert^2=:\frac{1}{2}\lVert y_{\rm obs}-\cG(u) \rVert_{R}^2
\end{equation}
and the normalisation constant
\begin{equation}
Z=\int_{\R^{D}}\exp(-l(u;y_{\rm obs}))\,\pi_0({\rm d}u) < \infty.
\end{equation}
If the prior PDF $\pi_0$ is Gaussian with mean $\mu_0 \in \mathbb{R}^{D}$ and covariance matrix $P_0 \in \mathbb{R}^{D\times D}$,
then the posterior is absolutely continuous with respect to the Lebesgue measure on $\R^D$ with PDF
\begin{equation}
\pi_\ast (u) = \frac{1}{Z} \exp(-\Phi(u;y_{\rm obs})),
\end{equation}
where
\begin{equation} \label{eq:Phi}
\Phi(u;y_{\rm obs}):=l(u;y_{\rm obs})+\frac12 \|u-\mu_0\|_{P_0}^2.
\end{equation}
We write  $\Phi(u)$ for simplicity and ignore the dependence on the data $y_{\rm obs}$ from now on.
\end{example}

\noindent
The sampling methods considered in this paper are based on stochastic processes of $N$ interacting particles
moving in $\mathbb{R}^D$ with the property that the marginal distributions in each of the particles approximate
$\pi_\ast$ as $t\to \infty$. The position of the $i$th particle is denoted by $u^{(i)}\in \mathbb{R}^D$ and its value at time $t\ge 0$
by $u_t^{(i)}$, $i=1,\ldots,N$. For ease of reference, we collect all particle positions into the
$D\times N$-dimensional matrix
\begin{equation} \label{eq:collected_matrix}
U = \left( u^{(1)}, u^{(2)},\ldots,u^{(N)}\right) \in \mathbb{R}^{D\times N}.
\end{equation}
The interacting particle systems to be considered in this paper obey gradient-based stochastic evolution equations of the form
\begin{equation}\label{eq:gradient_methods}
{\rm d}{u}_t^{(i)} = -\cA (U_t)\,\nabla_{u^{(i)}} \cV (U_t )\,{\rm d}t + \Gamma (U_t ) \, {\rm d} W_t^{(i)}, \qquad
i = 1,\ldots,N.
\end{equation}
Specific choices for the potential $\cV:\R^{D\times N}\to\R$, the positive semi-definite matrix-valued $\cA (U)\in \mathbb{R}^{D\times D}$ and $\Gamma (U) \in \mathbb{R}^{D\times L}$ will be discussed below.  $L$ is a natural number with typically either $L = D$ or $L = N$. The $W_t^{(i)}$ denote independent $L$-dimensional standard Brownian motions and the It{\^o}
interpretation \cite{sr:P14} of the multiplicative noise term in (\ref{eq:gradient_methods})  is to be used.

The main algorithmic
contribution of this paper consists in developing a particular instance of (\ref{eq:gradient_methods}) with the following three
properties:
\begin{itemize}
\item[(i)] The product measure
\begin{equation}
\label{eq:extended target}
\pi_\ast^{(N)} (U) := \prod_{i=1}^N \pi_\ast \left( u^{(i)}\right)
\end{equation}
is invariant under (\ref{eq:gradient_methods}). Furthermore, $\pi_\ast^{(N)}$ is ergodic in the sense that the joint law of the process converges towards $\pi_\ast^{(N)}$ as $t \rightarrow \infty$, in an appropriate sense and under suitable conditions on the initialisation. See \cite{sr:P14} for an introduction to ergodicity in the context of stochastic evolution equations.
\item[(ii)] The equations (\ref{eq:gradient_methods}) are invariant under affine transformations of the state variables, that is, for
transformations of the form
\begin{equation} \label{eq:affine_transform}
u = Mv + b
\end{equation}
for any invertible  $M\in \R^{D\times D}$ and any shift vector $b \in \R^D$. A precise definition of affine invariance is provided in
Definition \ref{def:affine_invariance} below. See also \cite{sr:GW10,sr:Greengard2015,sr:LMW18}.
\item[(iii)] The equations (\ref{eq:gradient_methods}) are straightforward and computationally efficient to implement, that is, do not require the inversion or factorisation of $D$-dimensional matrices and/or higher-order derivatives of the potential $\cV$.
\end{itemize}

\begin{definition}[Affine invariance]\label{def:affine_invariance}
Following \cite{sr:GW10,sr:Greengard2015,sr:LMW18}, a formulation (\ref{eq:gradient_methods}) is called \emph{affine invariant} 
under transformations of the form (\ref{eq:affine_transform}), that is,
\begin{equation}
u^{(i)} = Mv^{(i)} + b,
\end{equation}
if the resulting equations in the transformed particle positions are given by
\begin{equation}
{\rm d}v_t^{(i)} = -\cA (V_t)\,\nabla_{v^{(i)}} \widetilde{\cV} (V_t )\,{\rm d}t + \Gamma (V_t) \, {\rm d} W_t^{(i)}, \qquad
i = 1,\ldots,N,
\end{equation}
for any invertible matrix $M \in \mathbb{R}^{D\times D}$ and any shift vector $b \in \mathbb{R}^D$. 
Here 
\begin{equation}
V = \left( v^{(1)},v^{(2)},\ldots,v^{(N)}\right) \in \mathbb{R}^{D\times N},
\end{equation}
and the potential $\widetilde{\cV}$ is defined by
\begin{equation}
\widetilde{\cV} (V) = \cV (U) = \cV (MV + b\,1_N^{\rm T}) ,
\end{equation}
where $1_N\in \R^N$ denotes a column vector of ones.
\end{definition}

\begin{example}[Langevin dynamics] \label{ex:BD}
The classical example of (\ref{eq:gradient_methods}) is provided by the scaled first-order (overdamped) Langevin dynamics
\begin{equation} \label{eq:BD}
{\rm d}u_t^{(i)} = -C \,\nabla_{u^{(i)}} \Phi \left(u_t^{(i)} \right) \,{\rm d}t + \sqrt{2} C^{1/2} {\rm d}W_t^{(i)},
\end{equation}
where $W_t^{(i)}$, $i=1,\ldots,N$, denotes independent $D$-dimensional Brownian motion,
$C \in \R^{D\times D}$ is a constant symmetric positive-definite matrix and $C^{1/2}$ denotes its symmetric
positive-definite square root. In this case, the particles do not interact and $\cA = C$.
Furthermore, $\Gamma = \sqrt{2} C^{1/2}$ and the potential $\cV$ is given by
\begin{equation} \label{eq:standard_potential}
\cV (U) = \sum_{i=1}^N \Phi \left(u^{(i)} \right).
\end{equation}
We note that (\ref{eq:BD}) satisfies items (i) and (iii) from above for any $N\ge 1$ but not (ii), in general. As pointed out in \cite{sr:LMW18}, the failure of \eqref{eq:BD} to be affine invariant potentially leads to inefficient sampling when $\Phi$ is poorly scaled with respect to $C$. More specifically, in the case of Bayesian inverse problems with Gaussian posterior, this scenario occurs when $C$  is vastly different from the target covariance.
\end{example}

\noindent
Let $\pi_t^{(i)}$ denote the PDF of the $i$th particle $u_t^{(i)}$ at time $t\ge 0$  with evolution equation
(\ref{eq:BD}). Then these PDFs satisfy the Fokker--Planck equation
\begin{equation} \label{eq:FPE}
\partial_t \pi_t = \nabla_u \cdot \left( \pi_t \,C \,\nabla_u \frac{\delta {\rm KL}(\pi_t| \pi_\ast)}{\delta \pi_t}\right),
\end{equation}
with $\pi_t = \pi_t^i$ and  the Kullback--Leibler divergence defined by
\begin{equation} \label{eq:KLD}
{\rm KL}(\pi| \pi_\ast) =  \int_{\R^D} \log \left(\frac{\pi(u)}{\pi_\ast(u)}\right) \pi({\rm d}u).
\end{equation}
It has been shown in \cite{sr:JKO98} that the Fokker--Planck equation (\ref{eq:FPE})
corresponds to a gradient flow structure in the space of probability measures. Furthermore,
since the variational derivative of the Kullback--Leibler divergence is given by
\begin{equation}
\frac{\delta {\rm KL}(\pi_t| \pi_\ast)}{\delta \pi_t} = \log \pi_t - \log \pi_\ast,
\end{equation}
the invariance of the product measure (\ref{eq:extended target}) under the stochastic evolution equations (\ref{eq:BD}) follows
immediately.

An important generalisation of the linear Fokker--Planck equation (\ref{eq:FPE}) has been proposed
in \cite{sr:GIHLS19}. It relies on making the matrix $C$ dependent on the PDF $\pi_t$ itself; thus leading to a nonlinear
generalisation of (\ref{eq:FPE}). More specifically, the nonlinear Fokker--Planck equation is given by
\begin{equation} \label{eq:NLFPE}
\partial_t \pi_t = \nabla_u \cdot \left( \pi_t \,C(\pi_t) \,\nabla_u \frac{\delta {\rm KL}(\pi_t| \pi_\ast)}{\delta \pi_t}\right),
\end{equation}
with
\begin{equation}\label{eq:theoretical_cov}
C(\pi_t) = \E_{\pi_t} \left[(u-\mu_t)(u-\mu_t)^{\rm T}\right], \qquad \mu_t = \E_{\pi_t}\left[u\right].
\end{equation}
This choice of $C$ is motivated by the ensemble Kalman--Bucy filter \cite{sr:reich10,sr:cotterreich,sr:GIHLS19}. The associated generalised
gradient flow structure in the space of probability measures was first stated in \cite{sr:cotterreich} in the context of the ensemble
Kalman--Bucy filter mean-field equations and has been discussed in detail under the notion of the so-called Kalman--Wasserstein
gradient flow structure in \cite{sr:GIHLS19}. See Section \ref{sec:history} and Remark \ref{rem:EnKBF_GFS} below for more details.

A key observation for the present paper is that, contrary to the classical Fokker--Planck
equation (\ref{eq:FPE}) with constant $C$, the nonlinear Fokker--Planck equation (\ref{eq:NLFPE})
is affine invariant.


\begin{lemma}[Affine invariance of Kalman--Wasserstein dynamics] \label{lem:affine_invariance}
The nonlinear Fokker--Planck equation (\ref{eq:NLFPE}) is affine invariant.
\end{lemma}

\begin{proof}
We define the pushforward PDFs
\begin{equation}
\label{eq:pi transform}
\widetilde{\pi}_t(v) = |M|\, \pi_t(Mv+b), \qquad \widetilde{\pi}_\ast(v) = |M|\,\pi_\ast(Mv+b).
\end{equation}
Then
\begin{subequations}
\begin{align}
\partial_t \widetilde{\pi}_t &= |M|\, \partial_t \pi_t \\
&= |M|\,  \nabla_u \cdot \left( \pi_t \,C(\pi_t) \,\nabla_u \frac{\delta {\rm KL}(\pi_t| \pi_\ast)}{\delta \pi_t}\right)\\
&= \nabla_v \cdot \left( \widetilde{\pi}_t \,C(\widetilde{\pi}_t) \,\nabla_v \frac{\delta {\rm KL}(\widetilde{\pi}_t| \widetilde{\pi}_\ast)}{\delta
\widetilde{\pi}_t}\right).
\end{align}
\end{subequations}
Here we have used that
\begin{equation}
\label{eq:C transform}
C(\widetilde{\pi}_t) = M\, C(\pi_t)\,M^{\rm T},
\end{equation}
as well as $\nabla_v \widetilde{f}(v) = \nabla_v f(Mv+b) = M^{\rm T} \nabla_u f(u)$ for functions $\widetilde{f}(v) = f(u) = f(Mv+b)$
and an analog statement for the divergence operator. Furthermore, the variational derivatives of the Kullback--Leibler divergences satisfy
\begin{equation} \label{eq:transformed_KL}
\frac{\delta {\rm KL}(\pi_t| \pi_\ast)}{\delta \pi_t} =  \log \left(\frac{\pi}{\pi_\ast}\right)   =
\log \left(\frac{\widetilde{\pi}}{\widetilde{\pi}_\ast}\right)  = \frac{\delta {\rm KL}(\widetilde{\pi}_t| \widetilde{\pi}_\ast)}{\delta \widetilde{\pi}_t}.
\end{equation}
\end{proof}

\noindent
Building upon the affine invariance property of the nonlinear Fokker--Planck equation (\ref{eq:NLFPE}), we demonstrate in the following section how to obtain stochastic evolution equations of the form (\ref{eq:gradient_methods}) which satisfy all three properties (i)--(iii) from above. Their theoretical properties are studied in
the subsequent Section \ref{sec:properties}. In particular, we establish non-degeneracy and ergodicity, which provides the key
theoretical contribution of this paper.

%
\section{Affine invariant interacting Langevin dynamics} \label{sec:Algorithms}
%
As noted in the previous section, the nonlinear Fokker-Planck evolution \eqref{eq:NLFPE}-\eqref{eq:theoretical_cov} satisfies invariance of
the target measure ${\pi}^{(N)}_\ast$ (property (i)) as well as affine invariance (property (ii)). In this section, we address (iii), that is,
we present an interacting particle system of the form \eqref{eq:gradient_methods} which has \eqref{eq:NLFPE} as its mean field limit while
still maintaining properties (i) and (ii) for any finite number of particles. We also introduce a gradient-free approximation which is applicable to BIPs of the form (\ref{eq:Phi}). This section concludes with a summary of related previous algorithmic work.

%
\subsection{ALDI: An exact gradient-based sampling method} \label{sec:ALDI}
%

In order to define our interacting particle system, let us first define the empirical covariance matrix
\begin{equation}
\cC\left(U\right) := \frac{1}{N}\sum_{i=1}^N \left(u^{(i)} - m(U)\right)\left(u^{(i)} -
m(U)\right)^\mathrm{T}
\end{equation}
with empirical mean
\begin{equation}
m(U)  := \frac{1}{N}\sum_{i=1}^N u^{(i)} = \frac{1}{N} U\,1_N,
\end{equation}
that is, the particle-based estimators of the quantities defined in \eqref{eq:theoretical_cov}. We also introduce the
$D\times N$ matrix of the deviations of the particle positions from their mean value, that is
\begin{equation}
U' :=  \left( u^{(1)}-m(U),u^{(2)}-m(U),\ldots,u^{(N)}-m(U)\right) = U - m(U)\,1_N^{\rm T} ,
\end{equation}
which allows us to write
\begin{equation}
\cC \left(U\right) = \frac{1}{N} U' (U')^{\rm T}.
\end{equation}
Furthermore, we define a generalised (non-symmetric) square root of $\cC(U)$ via
\begin{equation}
\label{eq:square root}
\cC^{1/2}(U) := \frac{1}{\sqrt{N}} U' ,
\end{equation}
that is $\cC = \cC^{1/2} \left( \cC^{1/2}\right)^{\rm T}$.
For a moment, let us assume that $U \in \mathbb{R}^{D\times N}$ is such that $\cC(U)$ is invertible
(we will comment on this assumption following Definition \ref{def:ALDI}, see also Proposition \ref{prop:nondegeneracy})
and choose the preconditioning matrix
\begin{equation}
\label{eq:A}
\mathcal{A} (U) = \cC (U) = \frac{1}{N} U'(U')^{\rm T},
\end{equation}
the potential
\begin{equation}
\label{eq:potential}
\cV (U) = \sum_{i=1}^N \Phi \left(u^{(i)}\right) - \frac{D + 1}{2} \log | \mathcal{C}(U) |,
\end{equation}
and the diffusion matrix
\begin{equation}
\label{eq:diffusion matrix}
\Gamma (U) = \sqrt{2}\, \cC^{1/2} (U) = \frac{\sqrt{2}}{\sqrt{N}} U'
\end{equation}
in (\ref{eq:gradient_methods}), that is, $L=N$. Note that the potential (\ref{eq:potential}) contains the additional
$-(D+1)/2 \log |\cC(U)|$ term in comparison to (\ref{eq:standard_potential}), which is required to keep the target distribution 
(\ref{eq:extended target}) invariant under the state-dependent diffusion matrix $\cC(U)$. 
See Proposition \ref{prop:linear FPE} below and \cite{sr:NR19} for details.

Using the identity
\begin{equation} \label{eq:drift_identity}
\cC (U) \nabla_{u^{(i)}} \log | \cC (U)|   = \frac{2}{N} \left( u^{(i)} - m(U)\right),
\end{equation}
which follows from Jacobi's formula for the derivative of determinants (see the Appendix 
for more details), we derive the following explicit form of the proposed interacting particle Langevin dynamics.

\begin{definition}[ALDI method]
\label{def:ALDI}
The affine invariant Langevin dynamics (ALDI) is given by the interacting particle system
\begin{equation}
\label{eq:SDE}
\mathrm{d}u^{(i)}_t = - \cC (U_t ) \nabla_{u^{(i)}} \Phi \left(u_t^{(i)}\right)  \mathrm{d}t +
\frac{D+1}{N} \left(u_t^{(i)} -m(U_t) \right)  \mathrm{d}t + \sqrt{2} \,\cC^{1/2} (U_t) \,
 \mathrm{d}W_t^{(i)},
\end{equation}
for $i=1,\ldots,N$, where $W_t^{(i)}$ denotes $N$-dimensional standard Brownian motion.
\end{definition}

\noindent
We emphasise that the generalised square root $\cC^{1/2}(U)$, as defined in \eqref{eq:square root}, does not require a computationally
expensive Cholesky factorisation of $\cC(U)$, and hence the formulation \eqref{eq:SDE} satisfies the
requirement (iii). Note that although defining $\cV$ as in \eqref{eq:potential} necessitates  $N > D$ in order for the empirical covariance matrix $\cC(U)$ to be non-singular, the terms in \eqref{eq:SDE} are well-defined also for $N \le D$. While a non-singular $\cC(U)$ is required generically for the ALDI method to sample from the desired target measure $\pi_\ast^{(N)}$ (see the discussion in Section \ref{sec:ergodicity}), a smaller number of particles, $N$, is sometimes desirable in order to reduce the computational cost for high-dimensional BIPs.

If indeed $N\le D$, than $\cC(U)$ is singular and the dynamics of the interacting particle system
(\ref{eq:SDE}) is restricted to the linear subspace spanned by the $N$ initial particle positions $u_0^{(i)}$, that is,
\begin{equation} \label{eq:transform}
u^{(i)}_t = \sum_{j=1}^N m_t^{ij} u_0^{(j)} .
\end{equation}
Stochastic differential equations in the $N^2$ scalar coefficients $m_t^{ij}$ can easily be derived from (\ref{eq:SDE}) using the {\it ansatz}
(\ref{eq:transform}). In other words, provided that the initial samples $u_0^{(i)}$ are appropriately chosen, an implementation of
(\ref{eq:SDE}) with $N\le D$ can lead to a computationally efficient reduction of the BIP onto a lower dimensional linear subspace.
The affine invariance of (\ref{eq:SDE}) holds regardless of the ensemble size and is discussed in Section \ref{sec:AI} in more detail.


%
\subsection{Approximate gradient-free sampling} \label{sec:gradient_free}
%

A central idea put forward in \cite{sr:GIHLS19} (see also \cite{SPSR2019}) in the context of BIPs
described in Example \ref{ex:BIP} is to combine the preconditioned Langevin dynamics
with gradient-free formulations of the ensemble Kalman filter.
Recalling the forward map $\cG$ from \eqref{eq:IP}, the empirical cross-correlation matrix
$\cD(U)\in \R^{D\times K}$ is defined via
\begin{equation}
\cD(U) = \frac{1}{N}\sum_{i=1}^N \left(u^{(i)} - m(U)\right)\left(\mathcal{G}(u^{(i)}) -m(\cG(U))\right)^\mathrm{T}
\end{equation}
with empirical mean
\begin{equation}
m(\cG(U)) = \frac{1}{N} \sum_{i=1}^N \cG(u^{(i)}) = \frac{1}{N} \cG(U) \,1_N.
\end{equation}
We now make the approximation $\cC(U) \nabla_u \cG(u) \approx \cD(U)$, motivated by the fact that this approximation becomes
exact for affine forward maps, $\cG(u) = Gu + c$. We refer to \cite[Appendix A.1]{sr:evensen} for more details. In terms of  the ALDI formulation (\ref{eq:SDE}) and the potential $\Phi(u)$, given by (\ref{eq:Phi}), we obtain:

\begin{definition}[gradient-free ALDI]
Given a potential $\Phi (u)$ of the form (\ref{eq:Phi}), the gradient-free ALDI formulation is given by the interacting particle system
\begin{subequations}
\label{eq:SDE_DF}
\begin{align}
\mathrm{d}u^{(i)}_t &= - \left\{ \cD (U_t) R^{-1}\left(\mathcal{G}\left(u_t^{(i)}\right)-y_{\rm obs}\right) + \cC (U_t) P^{-1}_0(u_t^{(i)}-\mu_0) \right\}  \mathrm{d}t \\
\label{eq:repulsion gradient free}
& \qquad +\,\,
\frac{D+1}{N} \left(u_t^{(i)} - m(U_t) \right)  \mathrm{d}t + \sqrt{2} \cC^{1/2} (U_t)
\, \mathrm{d}W_t^{(i)},
\end{align}
\end{subequations}
for $i=1,\ldots,N$, where $W_t^{(i)}$ denote independent $N$-dimensional standard Brownian motions.
\end{definition}

\noindent
While the invariance of $\pi_\ast^{(N)}$ is lost under the gradient-free formulation (\ref{eq:SDE_DF}) (except, of course, for affine forward operators), affine invariance of the equations of motions is maintained; see Section \ref{sec:AI}.

%
\subsection{Related previous algorithmic work} \label{sec:history}
%

The idea of an affine invariant Monte Carlo method based on Langevin dynamics using an ensemble of particles and
its empirical covariance first appeared  in the unpublished Master thesis \cite{sr:Greengard2015}. More specifically, the author
proposes an affine invariant modification of the popular MALA algorithm \cite{sr:RS03,sr:GC11}, where  each particle
$u_k^{(i)}$, $i=1,\ldots,N$, is sequentially updated at time-step $k$ using the proposal
\begin{equation}
u_{k+1}^{(i)} = u_k^{(i)} - hM_k^{(i)} \nabla_{u_k^{(i)}} \Phi(u_k^{(i)}) + \sqrt{2h} L_k^{(i)} \xi_k^{(i)}
\end{equation}
where $h>0$ is the step-size, $M_k^{(i)}$ is an empirical covariance matrix based on a set of particles not including $u_k^{(i)}$, $L_k^{(i)}$
is the Cholesky factor of $M_k^{(i)}$, that is, $M_k^{(i)} = L_k^{(i)}(L_k^{(i)})^{\rm T}$, and $\xi_k^{(i)}$
is a $D$-dimensional Gaussian random variable with mean zero
and covariance matrix $I_{D\times D}$. Independently of \cite{sr:Greengard2015},  a general time-continuous framework for
affine invariant interacting particle formulations has been developed in \cite{sr:LMW18} and affine invariant implementations of
second-order Langevin dynamics using empirical covariance matrices are studied in detail.

More recently, ensemble preconditioned first-order Langevin dynamics has been revisited in \cite{sr:GIHLS19} with an emphasis on its
mean-field limit and its connection to the ensemble Kalman filter \cite{sr:evensen,sr:stuart15,sr:reichcotter15}. In fact, (\ref{eq:SDE})
appeared first in \cite{sr:GIHLS19} with the potential (\ref{eq:potential}) replaced by (\ref{eq:standard_potential}), that is
without the correction term
\begin{equation} \label{eq:correction_term}
\frac{D+1}{N} \left(u_t^{(i)} -m(U_t) \right),
\end{equation}
and with $\cC^{1/2}(U)$ being replaced by the symmetric matrix square root of the covariance matrix $\cC(U)$. The resulting method
is called the ensemble Kalman sampler (EKS) in \cite{sr:GIHLS19}.
The correction term (\ref{eq:correction_term}) is, however, needed in (\ref{eq:SDE}) in order for $\pi_\ast^{(N)}$ to be an invariant
distribution under the resulting interacting particle system (\ref{eq:gradient_methods}) and first appeared in the unpublished
note \cite{sr:NR19}. The invariance of $\pi_\ast^{(N)}$ under (\ref{eq:SDE}) is proven in Section \ref{sec:ergodicity}.

The correction term (\ref{eq:correction_term}) vanishes as $N\to \infty$ for $D$ fixed which
justifies the nonlinear Fokker--Planck equation (\ref{eq:NLFPE})  in this mean-field limit.
See \cite{sr:GIHLS19} for more details.

We note that a general discussion on necessary correction terms for Langevin dynamics with multiplicative noise can, for example, be found in \cite{sr:RS03,sr:GC11} from the perspective of Riemannian Brownian motion. We also note that general conditions on diffusion processes that guarantee invariance of a given target distribution have been investigated  in \cite[Section 2.2]{duncan2017using} and \cite{ma2015complete,sr:LMW18}.

The gradient-free approximation of the form $\cC(U) \nabla_u \cG(u) \approx \cD(U)$ originated in the ensemble
Kalman filter literature \cite{sr:evensen}. More precisely, the time continuous formulation of the ensemble 
Kalman filter, the so-called ensemble Kalman--Bucy filter given by
\begin{equation} \label{eq:EnKBF}
{\rm d} u_t^{(i)} = - \cC(U_t) \nabla_{u^{(i)}} \cG\left(u_t^{(i)}\right)R^{-1}\left( \frac{1}{2}\left\{ \mathcal{G}\left(u_t^{(i)}\right)+ m(\cG(U_t))\right\} -y_{\rm obs}\right) ,
\end{equation}
fits into the interacting particle dynamics framework (\ref{eq:gradient_methods}) with $\cA(U) = \cC(U)$, $\Gamma(U) = 0$, and
\begin{equation} \label{eq:EnKBF_potential}
\cV(U) = \frac{1}{4} \sum_{i=1}^N \Big\lVert y_{\rm obs}-\cG\left(u^{(i)}\right) \Big\rVert_{R}^2 + \frac{1}{4} \lVert y_{\rm obs}-m(\cG(U)) \rVert_{R}^2.
\end{equation}
See \cite{sr:reich10,sr:cotterreich,sr:reichcotter15} for more details. Its gradient-free formulation becomes
\begin{equation} \label{eq:EnKBF_DF}
{\rm d} u_t^{(i)} = - \cD(U_t) R^{-1}\left( \frac{1}{2}\left\{ \mathcal{G}\left(u_t^{(i)}\right)+ m(\cG(U_t))\right\} -y_{\rm obs}\right)
\end{equation}
\cite{sr:br11,sr:reichcotter15}. The derivative-free ensemble Kalman inversion (EKI) method \cite{sr:SS17,KovachkiStuart2018_ensemble}
is a slight modification of (\ref{eq:EnKBF_DF}) with the mean contribution $m(\cG(U))$ replaced by $\cG\left(u_t^{(i)}\right)$. This modification leads to a faster decay in the ensemble deviations $U_t'$ and, hence, in the covariance matrix $\cD(U_t)$ 
while retaining the evolution equation in the ensemble mean $m(U_t)$.

The extension of such gradient-free formulations to Langevin dynamics has been proposed first
in \cite{sr:GIHLS19}. Gradient-free formulations have been found to work well for unimodal posterior distributions in \cite{sr:GIHLS19}, but fail for multi-modal distributions as demonstrated in \cite{sr:RW19}. A localised covariance formulation of ALDI has been proposed in \cite{sr:RW19} to overcome this limitation. Localised covariance matrices were already considered in \cite{sr:LMW18}; but not in the context of gradient-free formulations.

%
\section{Theoretical analysis of ALDI} \label{sec:properties}
%

The aim of this section is to analyse some of the properties of the dynamics \eqref{eq:SDE}, in particular verifying the conditions (i) and (ii) outlined in Section \ref{sec:math prob}. The key observation (crucially depending on the correction term $\frac{D + 1}{2} \log | \mathcal{C}(U)|$ to the potential $\cV$ in \eqref{eq:potential}) is that the corresponding Fokker--Planck equation has the same mathematical structure as its counterpart \eqref{eq:NLFPE} for the mean-field regime:
\begin{proposition}[Linear Fokker--Planck equation]
\label{prop:linear FPE}
Let $U_t$, as defined by (\ref{eq:collected_matrix}),
satisfy the stochastic evolution equations (\ref{eq:SDE}) and assume that the time-marginal PDF ${\pi}_t^{(N)}$
of $U_t$ is smooth. Then $\pi_t^{(N)}$ satisfies the linear Fokker--Planck equation
\begin{equation} \label{eq:FPE_linear}
\partial_t \pi_t^{(N)} = \sum_{i=1}^N \nabla_{u^{(i)}} \cdot \left( \pi_t^{(N)} \,\cC \,\nabla_{u^{(i)}} \frac{\delta {\rm KL} \left(\pi_t^{(N)}| \pi_\ast^{(N)} \right)}{\delta \pi_t^{(N)}}\right).
\end{equation}
\end{proposition}
\begin{proof}
The proof can be found in the Appendix. 
See also the technical report \cite{sr:NR19}.
\end{proof}

\noindent
Note that the PDF $\pi_t^{(N)}$ in \eqref{eq:FPE_linear} is defined on the extended space $\mathbb{R}^{D\times N}$,
whereas $\pi_t$ in \eqref{eq:NLFPE} is defined on $\mathbb{R}^D$. In contrast to \eqref{eq:FPE_linear}, the mean-field equation
\eqref{eq:NLFPE} is nonlinear since $C(\pi_t)$ depends on the solution $\pi_t$ itself.

%
\subsection{Non-degeneracy and ergodicity}
\label{sec:ergodicity}
%

As a first result, we have that property (i) is satisfied for the extended target measure \eqref{eq:extended target}
on the joint state space $\mathbb{R}^{D\times N}$. This follows directly from Proposition \ref{prop:linear FPE}:
\begin{corollary}[Invariance of the posterior measure]
	\label{cor:invariance}
	The extended target measure (\ref{eq:extended target})
	is invariant for \eqref{eq:SDE}, that is, if $U_0 \sim \pi_\ast^{(N)}$,
	then $U_t \sim \pi_\ast^{(N)}$ for all $t \ge 0$.
\end{corollary}
\begin{proof}
Observe that $\mathrm{KL}( \pi^{(N)} \vert \pi_\ast^{(N)})$ is minimised for $\pi^{(N)} = \pi_\ast^{(N)}$, and hence
$$
\frac{\delta {\rm KL}\left(\pi^{(N)}| \pi_\ast^{(N)} \right)}{\delta \pi^{(N)}} \Big\vert_{\pi^{(N)} = \pi_\ast^{(N)}} = 0.
$$
Using \eqref{eq:FPE_linear}, we immediately see that $\partial_t \pi_\ast^{(N)} = 0$, implying the claimed result.
\end{proof}


\noindent Note that $\pi_\ast^{(N)}$ is not the unique invariant measure for the dynamics \eqref{eq:SDE}. For instance, if $U = \left(u^{(1)}, \ldots, u^{(N)}\right)$ with $u^{(1)} = u^{(2)} = \ldots = u^{(N)}$, then $\cC(U) = 0$ and $u^{(i)} = m(U)$, and hence $\delta_U$ (the Dirac measure centred at $U$) is invariant.
To ensure favourable ergodic properties, we need to prove that $\pi_\ast^{(N)}$ is the unique invariant measure that is reachable by the dynamics from an appropriate set of initial conditions. First, we shall make the following assumption on the
potential $\Phi$:
\begin{assumption}[Regularity and growth conditions on the potential $\Phi$]
	\label{ass:V}
	Assume that $\Phi \in C^2(\mathbb{R}^D) \cap L^1(\pi_\ast)$. Furthermore, assume that there exists a compact set $K \subset \mathbb{R}^D$ and constants $c_2>c_1 > 0$ such that
	\begin{subequations}
		\begin{align}
		c_1 \vert u \vert^2 & \le \Phi(u) \le c_2 \vert u \vert^2, \\
		c_1 \vert u \vert  & \le \vert \nabla \Phi(u) \vert \le c_2 \vert u \vert, \\
		\label{eq:Hess bound}
		c_1  I_{D \times D} & \le  \Hess \Phi(u) \le c_2 I_{D \times D},
		\end{align}
	\end{subequations}
	for all $u \in \mathbb{R}^D \setminus K$.
\end{assumption}

\noindent
The bound \eqref{eq:Hess bound} is to be understood in the sense of quadratic forms. Assumption \ref{ass:V} is satisfied for target measures with Gaussian tails. Indeed, $\Phi = \Phi_0 + \Phi_1$ is admissible, where $\Phi_0(u) = \frac{1}{2} u \cdot S u$ is quadratic (with $S \in \mathbb{R}^{D \times D}$ strictly positive definite), and $\Phi_1 \in C_c^\infty(\mathbb{R}^D)$ is a smooth perturbation with compact support. We would like to emphasise that Assumption \ref{ass:V} can be relaxed with minimal effort, but we refrain from doing so for ease of exposition.

Due to the fact that $\cC(U)$ is not uniformly bounded from below on $\mathbb{R}^{D\times N}$, the associated Fokker-Planck
operator is not uniformly elliptic and standard ergodicity results are not applicable.
However, we have the following non-degeneracy result.

\begin{proposition}[Non-degeneracy of the empirical covariance matrix]
	\label{prop:nondegeneracy}
	Let Assumption \ref{ass:V} be satisfied and assume that $\cC(U_0)$ is strictly positive definite. Then \eqref{eq:SDE} admits a unique global strong solution, and $\cC(U_t)$ stays strictly positive definite for all $t \ge 0$, almost surely.
\end{proposition}
\begin{proof}
	The proof rests on the identity (\ref{eq:drift_identity})
	so that \eqref{eq:SDE} can be written in the form
	\begin{equation}
	\label{eq:modified SDE}
	\mathrm{d}u^{(i)}_t = - \cC(U_t) \nabla_{u^{(i)}} \cV(U_t) \, \mathrm{d}t  + \sqrt{2} \cC^{1/2}(U_t)
	\, \mathrm{d}W_t^{(i)}, \quad i=1,\ldots,N,
	\end{equation}
	with the potential $\cV$ given by (\ref{eq:potential}), making use of the repulsive effect of the term $$- \frac{D + 1}{2} \log | \cC(U)|.$$ Details can be found in the Appendix. 
\end{proof}
\noindent With Proposition \ref{prop:nondegeneracy} in place, the proof of the following ergodicity result is relatively straightforward:
\begin{proposition}[Ergodicity]
	\label{prop:ergodicity}
	Assume the conditions from Proposition \ref{prop:nondegeneracy}, and furthermore that $N > D+1$. Then the dynamics is ergodic, that is,
	${\pi}^{(N)}_t \rightarrow \pi^{(N)}_\ast$ as $t \rightarrow \infty$ in total variation distance.
\end{proposition}
\begin{proof}
The proof can be found in the Appendix. 
\end{proof}
\begin{remark}
In the case when $N\le D$, ergodicity will not hold, since the dynamics is constrained to a subspace according to the discussion following Definition \ref{def:ALDI}. In the case when $N=D+1$ one can show that the set 
\begin{equation}
\label{eq:invertible set}
E = \left\{ U \in \mathbb{R}^{D \times N}: \quad \cC(U)  \text{ is invertible}  \right\}
\end{equation}
has two connected components. The dynamics will then be ergodic with respect to $\pi^{(N)}_\ast$ restricted to one of these, depending on the initial condition. This is acceptable from an algorithmic viewpoint, but we do not treat this case separately for simplicity.
\end{remark}

%
\subsection{Affine invariance} \label{sec:AI}
%

We show that \eqref{eq:SDE} and its gradient-free variant (\ref{eq:SDE_DF}) are affine-invariant, in the terminology introduced in \cite{sr:GW10,sr:Greengard2015} and summarised in Definition \ref{def:affine_invariance}.

\begin{lemma}[Affine invariance of ALDI] \label{lem:affine_invariance_ALDI}
The Fokker--Planck equation (\ref{eq:FPE_linear}), its associated interacting particle system (\ref{eq:SDE}) as well
as its gradient-free formulation (\ref{eq:SDE_DF}) are all affine invariant.
\end{lemma}

\begin{proof}
We follow the proof of Lemma \ref{lem:affine_invariance}. Since $\cC(U) = M\,\cC(V)
M^{\rm T}$ we also have $\cA(U) = M \cA(V) M^{\rm T}$. Furthermore,
\begin{equation}
\nabla_{v^{(i)}} \widetilde{f}(V) = \nabla_{v^{(i)}} f \left(M V + b\,1_N^{\rm T}\right)
= M^{\rm T} \nabla_{u^{(i)}} f(U)
\end{equation}
for functions $ \widetilde{f}(V) =  f(U) = f\left(M V+ b\,1_N^{\rm T} \right)$,
and an analogous statement holds for the divergence operator. Finally, equality (\ref{eq:transformed_KL})
also holds for the Kullback--Leibler divergences over extended state space. Along the same lines, the affine invariance can also
be checked directly at the level of the stochastic differential equations (\ref{eq:SDE}). In particular, it holds that
$\cC^{1/2}(U) = M \cC^{1/2}(V)$. Furthermore,
\begin{equation}
\cD(U) = M \cD(V)
\end{equation}
with $\widetilde{\cG}(v) = \cG(Mu+b)$ and $\cD(V)$ the empirical covariance matrix between $v$ and $\widetilde{\cG}(v)$.
This implies the affine invariance of the gradient-free formulation (\ref{eq:SDE_DF}).
\end{proof}
\begin{remark}[Path-wise versus distributional affine invariance]
	Definition \ref{def:affine_invariance} is based on path-wise affine invariance
	at the level of the SDE (\ref{eq:gradient_methods}). Path-wise invariance implies affine invariance
	of the associated time-marginal distributions $\pi_t^{(N)}$, that is, affine invariance of the implied Fokker--Planck
	equation. The converse is not true, in general.
\end{remark}

%
%
\subsection{Geometric properties and gradient flow structure} \label{sec:gradient_flow_structure}
%
%


In this section, we place the dynamics \eqref{eq:SDE} in a geometric context, viewing (a suitable subset of) $\mathbb{R}^{D \times N}$ as a Riemannian manifold when equipped with an appropriate metric tensor. This approach has been pioneered in \cite{sr:GC11}; we also recommend the review paper \cite{livingstone2014information}. Leveraging this perspective, we show that the evolution induced by  \eqref{eq:SDE} on the set of smooth PDFs can be interpreted as a gradient flow in the sense of \cite{sr:JKO98}. In the limit as $N \rightarrow \infty$ we formally recover the Kalman--Wasserstein geometry introduced in \cite{sr:GIHLS19}.

We restrict our attention to the case $N > D + 1$ in this section, when the dynamics \eqref{eq:SDE} is ergodic on the set
$E$, as defined in \eqref{eq:invertible set}, according to Proposition \ref{prop:ergodicity}. Extending the framework to the case when $N \le D+1$ is subject of ongoing work. We now turn $E$ into a $D \times N$-dimensional Riemannian manifold. Denoting the $\gamma$-th coordinate of the $i$-th particle by $U^{(\gamma,i)}$, we introduce the metric tensor
\begin{equation}\label{eq:metric tensor}
	g  = \sum_{i=1}^N \sum_{\gamma,\sigma =1}^D \cC^{-1}_{\gamma \sigma} \mathrm{d}U^{(\gamma,i)}\, \mathrm{d}U^{(\sigma,i)},
\end{equation}



\noindent
In what follows, we will denote by $\mathrm{d}\text{vol}_g$ the Riemannian volume, by $\nabla_g$ the Riemannian gradient, by $(W^g_t)_{t \ge 0}$ Riemannian Brownian motion and by $d_g$ the geodesic distance on $(E,g)$.  For more details, we refer to \cite{hsu2002stochastic,lee2006riemannian} and, in the context of computational statistics, to \cite{livingstone2014information}. Using these objects induced by $g$, both the SDE \eqref{eq:SDE} and the corresponding Fokker--Planck equation \eqref{eq:FPE_linear} admit a compact formulation:
	\begin{proposition}[Riemmanian interpretation of ALDI]
	Let $\pi^{(N),g}$ denote the density of $\pi^{(N)}$ with respect to the Riemannian volume, that is,
	$\pi^{(N),g} \, \mathrm{d}\mathrm{vol}_g = \pi^{(N)}\, \mathrm{d} U$ and $\pi_\ast^{(N),g} \, \mathrm{d}\mathrm{vol}_g = \pi^{(N)}_\ast\,
	\mathrm{d} U$. Then the dynamics  \eqref{eq:SDE} can be written in the form
	\begin{equation}
	\label{eq:geometric formulation}
	\mathrm{d}U_t =  \nabla^g_U \log \pi_\ast^{(N),g}(U_t) \, \mathrm{d}t + \sqrt{2}\, \mathrm{d}W_t^g,
	\end{equation}
and the Fokker--Planck equation \eqref{eq:FPE_linear} can be written in the form
\begin{equation} \label{eq:FPE_geometric}
\partial_t \pi^{(N),g}_t = \nabla^g_U \cdot \left( \pi^{(N),g}_t \,\nabla^g_U
\frac{\delta {\rm KL}(\pi^{(N)}_t| \pi^{(N)}_\ast)}{\delta \pi^{(N)}_t}\right).
\end{equation}
\end{proposition}
\begin{remark}
Note that the Kullback--Leibler divergence and its functional derivative depend on the measures but not on the respective densities, in contrast to the Onsager operator \cite{machlup1953fluctuations,mielke2016generalization,ottinger2005beyond} $
\phi \mapsto -\nabla^g_U \cdot \left( \pi^{(N),g}_t \,\nabla^g_U  \phi \right)
$.
\end{remark}
\begin{proof}
Using the results from \cite{livingstone2014information}, in particular the equations (46)-(47),
the proof of the first statement reduces to verifying that
\begin{equation}
\label{eq:geometric correction}
\partial_J g^{IJ} = \frac{D+1}{N} \left(u^{(i)} - m(U)\right)_\gamma,
\end{equation}
where $g^{IJ}$ stands for the components of the inverse of $g$, and we have used the notation $I = (\gamma,i)$ and
$J=(\sigma,j)$. Furthermore, we apply  Einstein's summation convention here and in the remainder of this proof.
The statement \eqref{eq:geometric correction} follows directly from the definition of $g$ and the identity \cite{sr:NR19} 
\begin{equation}
\label{eq:div calculations}
\nabla_{u^{(i)}} \cdot \cC(U) = \frac{D+1}{N} (u^{(i)}-m(U))
\end{equation}
giving rise to the drift correction (\ref{eq:correction_term}). Indeed, together with the coordinate expressions
\begin{equation}
\label{eq:nabla formulas}
\nabla^g_U \cdot f = \frac{1}{\sqrt{\vert g \vert}} \partial_I \left(\sqrt{\vert g \vert} f^I \right), \quad (\nabla^g_U V)^I = g^{IJ} \partial_J V
\end{equation}
for vector-valued functions $f$ and scalar-valued $V$,
the result follows by direct substitution.

For the second statement, note that $\mathrm{d}\mathrm{vol}_g = \sqrt{\vert g \vert} \, \mathrm{d}U$, and  hence $\pi^{(N),g} = \vert g \vert^{-1/2} \pi^{(N)}$. \end{proof}

%
%
%
%

To exhibit the gradient flow structure, we recall that the natural quadratic Wasserstein distance between probability measures defined on $(E,g)$ is given by
\begin{equation}
\label{eq:W2_g}
\mathcal{W}_g^2\left(\mu^{(N)},\nu^{(N)}\right) = \inf_{\gamma \in \Pi \left(\mu^{(N)},\nu^{(N)} \right)} \int_{E \times E} d^2_g(U,V) \,
\mathrm{d}\gamma (U,V),
\end{equation}
where $\Pi\left(\mu^{(N)},\nu^{(N)}\right)$ denotes the set of probability measures on $E \times E$ with marginals $\mu^{(N)}$ and $\nu^{(N)}$. It is well-known that the evolution \eqref{eq:FPE_geometric} can be interpreted as gradient flow dynamics of the Kullback--Leibler divergence on the set of probability measures equipped with the distance \eqref{eq:W2_g}, see  for instance \cite[Chapter 15]{villani2008optimal} or \cite{lisini2009nonlinear}. By the Benamou--Brenier formula \cite{benamou2000computational}, we have the representation
\begin{subequations}
	\label{eq:Wg}
\begin{align}
\mathcal{W}_g^2\left(\mu^{(N)},\nu^{(N)}\right) = \inf_{\{\pi_t,\Phi_t\}} \Bigg\{ & \int_0^1 \int_{E} g(\nabla^g_U \Phi_t,\nabla^g_U \Phi_t) \,\mathrm{d}\pi_t \,\mathrm{d}t:
\\
\label{eq:cont eq}
 & \partial_t \pi_t^g + \nabla^g_U \cdot(\pi_t^g \, \nabla^g_U \Phi_t)  = 0,\quad  \pi_0 = \mu^{(N)}, \,\, \pi_1 = \nu^{(N)} \Bigg\},
\end{align}
\end{subequations}
where the constraining continuity equation in \eqref{eq:cont eq} is to be interpreted in a weak form and we again denoted by $\pi^g$ the density of $\pi$ with respect to $\mathrm{d}\mathrm{vol}_g$.
In standard coordinates (using the definition \eqref{eq:metric tensor} as well as the formulas \eqref{eq:nabla formulas}) we see that
\begin{subequations}
	\begin{align}
\label{eq:ALDI distance}
\mathcal{W}_g^2\left(\mu^{(N)},\nu^{(N)}\right) = \inf_{\{\pi_t,\Phi_t\}} \Bigg\{ & \int_0^1 \int_{E} \nabla_U \Phi_t \cdot \cC\, \nabla_U \Phi_t \,\mathrm{d}\pi_t \mathrm{d}t:
\\
 & \partial_t \pi_t + \nabla_U \cdot(\pi_t \cC\, \nabla_U \Phi_t)  = 0,\,  \quad \pi_0 = \mu^{(N)}, \, \pi_1 = \nu^{(N)} \Bigg\}, 
\end{align}
\end{subequations}
revealing a close similarity with the Kalman--Wasserstein distance (here denoted by $\mathcal{W}_{\mathrm{Kalman}}$) introduced in \cite{sr:GIHLS19}. Indeed, let us choose $\mu^{(N)} := \otimes_{i=1}^N \mu^{(i)}$ and $\nu^{(N)}:=\otimes_{i=1}^N \nu^{(i)}$, the product measures on $\mathbb{R}^{D \times N}$ associated to $\mu,\nu \in \mathcal{P}(\mathbb{R}^D)$, where $\mu^{(i)}$ and $\nu^{(i)}$, $i=1,\ldots,N$ are understood to be identical copies of $\mu$ and $\nu$. We formally expect that
\begin{equation}
\frac{1}{N}\mathcal{W}_g^2 \left( \mu^{(N)},  \nu^{(N)} \right) \xrightarrow{N\rightarrow \infty} \mathcal{W}_{\mathrm{Kalman}}(\mu,\nu),
\end{equation}
using that $\cC(U) \approx C(\pi)$ for sufficiently large $N$, where $C(\pi)$ was defined in \eqref{eq:theoretical_cov}. A rigorous passage from $\mathcal{W}_g$ to the Kalman--Wasserstein
distance might be a rewarding direction for future research; we note that a similar analysis (relating the gradient flow structures associated to a finite particle system and its mean-field limit) has been carried out recently in \cite{carrillo2019proof}.

\begin{remark}[Gradient flow structure of the ensemble Kalman--Bucy filter] \label{rem:EnKBF_GFS}
Taking the formal mean-field limit of the ensemble Kalman--Bucy filter (\ref{eq:EnKBF}) leads to the following evolution equation in the
marginal densities $\pi_t$:
\begin{equation} \label{eq:FPE_EnKBF}
\partial_t \pi_t = \nabla_u \cdot \left( \pi_t \,C(\pi_t) \,\nabla_u \frac{\delta \mathcal{F}_{\rm EnKBF}(\pi_t)}{\delta \pi_t}\right)
\end{equation}
with potential
\begin{equation} \label{eq:EnKBF_pot}
\mathcal{F}_{\rm EnKBF}(\pi) = \frac{1}{4} \int_{\mathbb{R}^D} \lVert y_{\rm obs}-\cG(u) \rVert_{R}^2\, \pi(u)\,{\rm d}u
+ \frac{1}{4} \lVert y_{\rm obs}-\mathbb{E}_{\pi}[\cG(u)] \rVert_{R}^2,
\end{equation}
which arises naturally from (\ref{eq:EnKBF_potential}) in the limit $N\to \infty$ \cite{sr:cotterreich,sr:reichcotter15}.
Note that (\ref{eq:FPE_EnKBF}) is exactly of the
form (\ref{eq:NLFPE}) with the Kullback--Leibler divergence being replaced by the potential (\ref{eq:EnKBF_pot}).
Its gradient flow structure in the space of probability measures has been first discussed in \cite{sr:cotterreich,sr:reichcotter15} and
is equivalent to the Kalman--Wasserstein gradient flow structure introduced in \cite{sr:GIHLS19}. The mean-field limit of
the EKI  \cite{sr:SS17,KovachkiStuart2018_ensemble} also fits within this framework with the potential $\mathcal{F}_{\rm EnKBF}(\pi)$ replaced by
\begin{equation}
\mathcal{F}_{\rm EKI}(\pi) = \frac{1}{2} \int_{\mathbb{R}^D} \lVert y_{\rm obs}-\cG(u) \rVert_{R}^2\, \pi(u)\,{\rm d}u .
\end{equation}
The affine invariance of both the EnKBF and EKI follows along the lines of Lemma \ref{lem:affine_invariance}. As for the finite ensemble
size formulations, one expects a slower decay of $\mathcal{F}_{\rm EnKBF}(\pi_t)$ compared to $\mathcal{F}_{\rm EKI}(\pi)$.
\end{remark}

%
%
%

%
%
\section{Numerical experiment: A PDE constrained inverse problem} \label{sec:numerics}
%
%

We consider the inverse problem of determining the permeability field $a(x)>0$ in the elliptic partial differential equation (PDE)
\begin{equation} \label{eq:model}
-\partial_x (a(x) \partial_x p(x)) = f(x), \qquad x \in \Omega= [0,2\pi),
\end{equation}
from $K=10$ observed grid values
\begin{equation} \label{eq:num_obs}
y_j = p({\rm x}_j) + \eta_j, \qquad {\rm x}_j = \frac{2\pi (j-1)}{K},
\end{equation}
$j=1,\ldots,K$, of the pressure field $p$ for a given forcing $f$. Both $p$ and $f$ are assumed to integrate to zero over the domain
$\Omega$. The measurement errors $\eta_j$ in (\ref{eq:num_obs}) are i.i.d.~Gaussian with mean zero
and variance $\sigma_R = 10^{-4}$.  A related 2-dimensional Darcy flow problem has been studied in \cite{sr:GIHLS19}.
In this paper, we restrict the simulations to the 1-dimensional formulation (\ref{eq:model}) for computational simplicity.

This infinite-dimensional problem is made finite-dimensional by introducing a computational grid
\begin{equation}
x_i = \frac{2\pi i}{D}, \qquad i =0,\ldots,D-1,
\end{equation}
with $D = 50$ grid points. Hence (\ref{eq:model}) gets replaced by the finite-difference formulation
\begin{equation} \label{eq:num_model}
\frac{a_{i+1/2}(p_{i+1}-p_i)-a_{i-1/2}(p_i-p_{i-1})}{h^2} = -f_i,
\end{equation}
$i=1,\ldots,D$. Here $h = 2\pi/D$ denotes the mesh size and $p_i \approx p(x_i)$, etc. We also make use of the
periodicity and set $p_D = p_0$ as well as $f_D=f_0$ .

Since the permeability field should be non-negative, we set
\begin{equation} \label{eq:num_truth0}
a_{i-1/2} = \exp ( u_i  )
\end{equation}
for $i=1,\ldots,D$. 
The computational forward problem is now given by the solution $\{p_i\}_{i=0}^{D-1}$ to (\ref{eq:num_model}) for given $\{f_i\}_{i=0}^{D-1}$
and $\{u_i\}_{i=1}^{D}$ and its restriction to the observation grid $\{{\rm x}_j\}_{j=1}^{K}$. 
We denote this map by $\cG(u)$, suppressing the dependence on the forcing given by
\begin{equation}
f_i = \exp \left(- \frac{(2x_i-L)^2}{40}\right)-c_f,
\end{equation}
where $c_f>0$ is chosen such that the forcing has mean zero.
The measurement error covariance matrix is given by $R = \sigma_R I_{K\times K}$. This completes
the description of our forward model (\ref{eq:IP}). 

The prior distribution on $u\in \mathbb{R}^D$ is assumed to be Gaussian with mean zero and covariance matrix $P_0$
defined by
\begin{equation}
P_0^{-1} = 4h \left(\frac{\mu}{D}1_D1_D^{\rm T} - \Delta_h \right)^2,
\end{equation}
where $\Delta_h$ denotes the standard second-order finite-difference operator 
over $\Omega$ with mesh-size $h$ and periodic  boundary conditions, that is, the operator defined by the left-hand side of (\ref{eq:num_model}) with $a_{i\pm1/2} = 1$. The parameter $\mu>0$ is set to $\mu=10^{2}$ leading to a penalty on
the (spatial) mean of $u= \{u_i\}_{i=1}^D$ to be close to zero.

The observations (\ref{eq:num_obs}) are generated numerically by solving (\ref{eq:num_model}) with the 
reference permeability field given by
\begin{equation} \label{eq:num_truth1}
a^\dagger_{i-1/2}  =  \exp ( u^\dagger_i  ), 
\end{equation}
where
\begin{equation}
u^\dagger_i = \frac{1}{2}\sin (x_i - h/2) 
\end{equation}
for $i=1,\ldots,D$, and setting
\begin{equation}
y_j = p_l + \eta_j, \qquad l = \frac{D}{K}\,j = 5j, \qquad \eta_j \sim \cN(0,\sigma_R),
\end{equation}
$j = 1,\ldots,K$. 

We implemented the gradient-based ALDI formulation (\ref{eq:SDE}) as well as the gradient-free ALDI formulation (\ref{eq:SDE_DF})
using the Euler--Maruyama method with step-size $\Delta t = 0.01$ over a time interval $t\in [0,20]$. In line with \cite{sr:GIHLS19} we refer to the ALDI implemented without the correction term (\ref{eq:correction_term}) as the ensemble Kalman sampler (EKS). The ensemble sizes were taken as $N = 25,52,100,200$. Except for the smallest ensemble size, all other choices resulted in non-singular empirical covariance matrices $\cC(U_t)$.

We compare the simulation results based on the estimation bias
\begin{equation} \label{eq:BIAS}
{\rm BIAS} = \frac{h}{T} \int_{\tau}^{\tau +T} \|m({U}_t) - u^\dagger\|^2 \,{\rm d}t
\end{equation}
and the ensemble spread
\begin{equation} \label{eq:SPREAD}
{\rm SPREAD} = \frac{h}{T} \int_\tau^{\tau +T} \mbox{trace} \left(\cC(U_t)\right)
\end{equation}
computed along numerical solutions for $\tau = 12$ and $T=8$. Each experiment was repeated ten times to reduce
the impact of random effects. The results can be found in Tables \ref{table1} and \ref{table2}, respectively. It can be
seen that the correction term has a profound impact on both the bias as well as the ensemble spread for the smallest
ensemble size $N=25$. This effect is largely diminished for the largest ensemble size of $N=200$.  We also find
that the gradient-free implementations yield results which are essentially indistinguishable from those based on the exact
gradient while being computationally much more efficient. Finally, the results for ALDI indicate that it is entirely
sufficient to implement it with $N = D+2 = 52$ particles; the minimum size required for ergodicity to hold.

\begin{table}[h!]
\begin{center}
\begin{tabular}{l|c|c|c|c}
\textbf{N} & \textbf{gf-EKS} & \textbf{gf-ALDI} & \textbf{g-EKS} & \textbf{g-ALDI}\\
\hline
25 & 0.5035 & 0.4113 & 0.4940 & 0.4042 \\
52 & 0.3748 & 0.3028 & 0.3706 & 0.2957 \\
100 & 0.3215 & 0.3070 & 0.3166 & 0.3016 \\
200 & 0.3088 & 0.3081 & 0.3030 & 0.3009
\end{tabular}
\end{center}
\caption{Computed estimation bias (\ref{eq:BIAS}) for ensemble sizes $N\in \{25,51,100,200\}$ and implementations of ALDI
and EKS as well as with exact gradient (g) and gradient-free (gf). } \label{table1}
\end{table}

\begin{table}[h!]
\begin{center}
\begin{tabular}{l|c|c|c|c}
\textbf{N} & \textbf{gf-EKS} & \textbf{gf-ALDI} & \textbf{g-EKS} & \textbf{g-ALDI}\\
\hline
25 & 0.0082 & 0.0724 & 0.0083 & 0.0738 \\
52 & 0.0135 & 0.0475 & 0.0134 & 0.0476 \\
100 & 0.0219 & 0.0457 & 0.0218 & 0.0457\\
200 & 0.0337 & 0.0453 & 0.0336 & 0.0453
\end{tabular}
\end{center}
\caption{As in Table \ref{table1}, but reporting the results for the ensemble spread (\ref{eq:SPREAD}).} \label{table2}
\end{table}

In order to provide a better insight into the impact of the correction term (\ref{eq:correction_term}) on the final
ensemble distributions we display results for $M=25$ and $M=200$ in Figures \ref{figure1} and \ref{figure2},
respectively.

\begin{figure}
\begin{center}
\includegraphics[width=0.45\textwidth,trim = 0 0 0 0,clip]{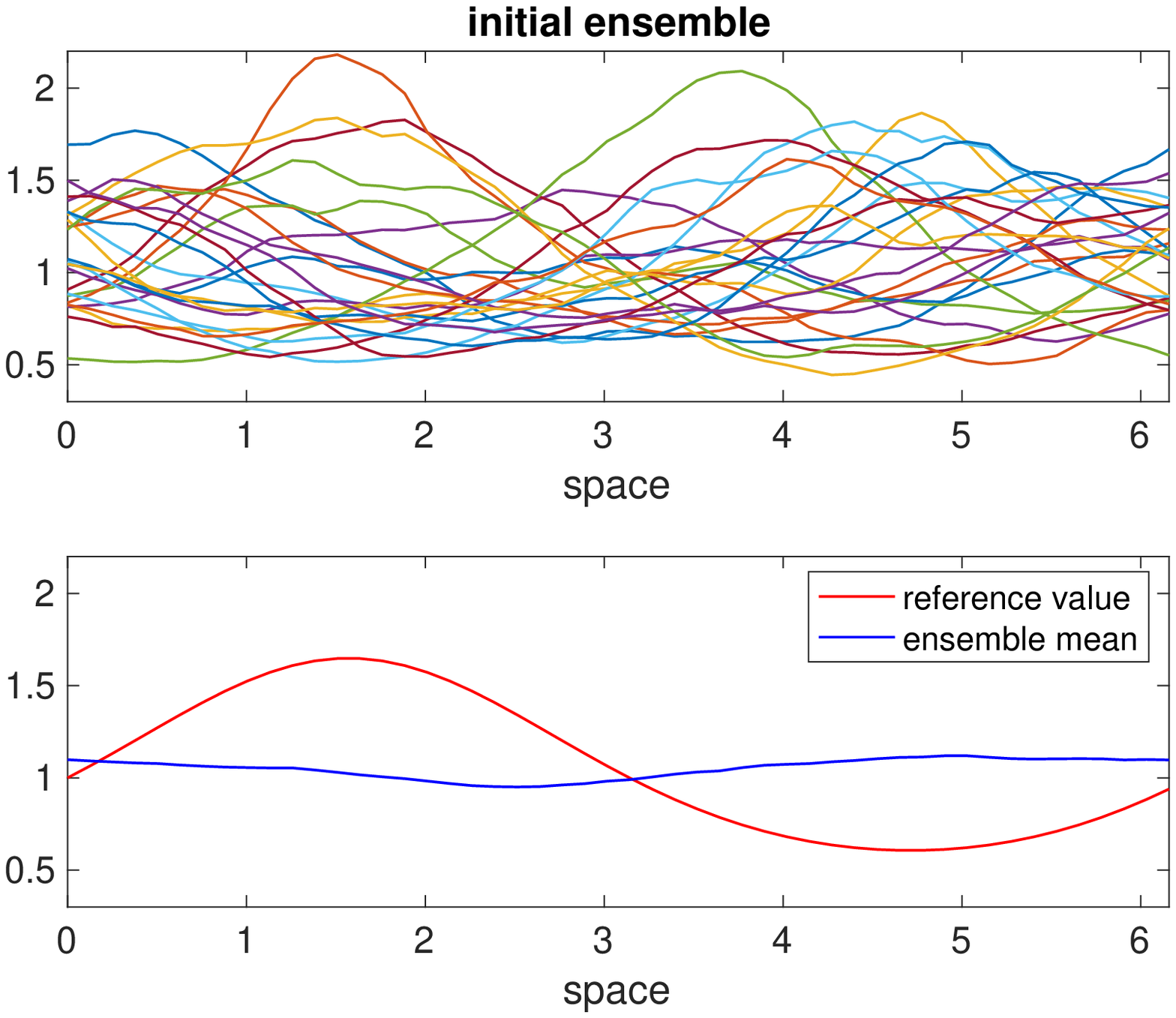} $\qquad$
\includegraphics[width=0.45\textwidth,trim = 0 0 0 0,clip]{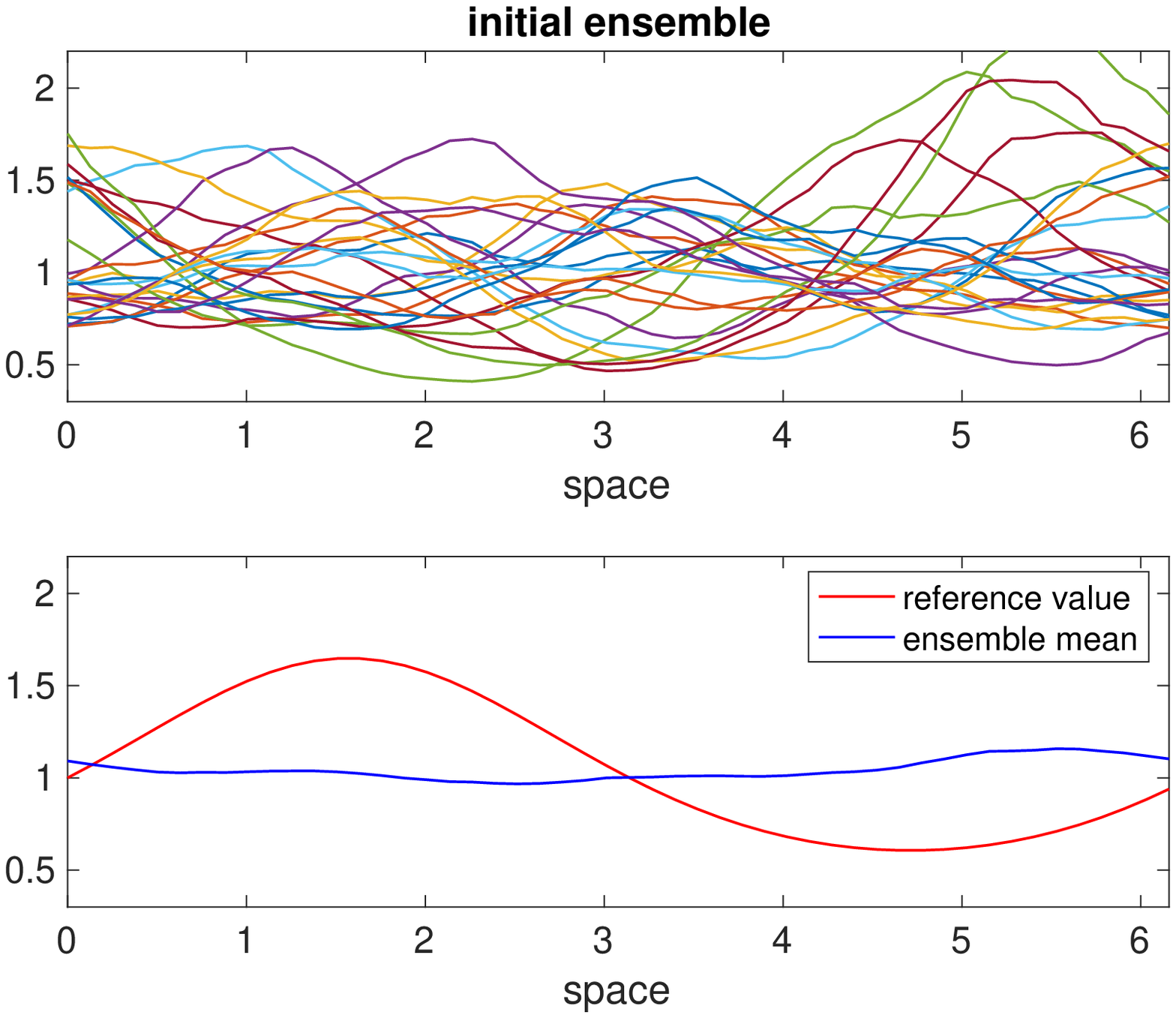} \\ \medskip
\includegraphics[width=0.45\textwidth,trim = 0 0 0 0,clip]{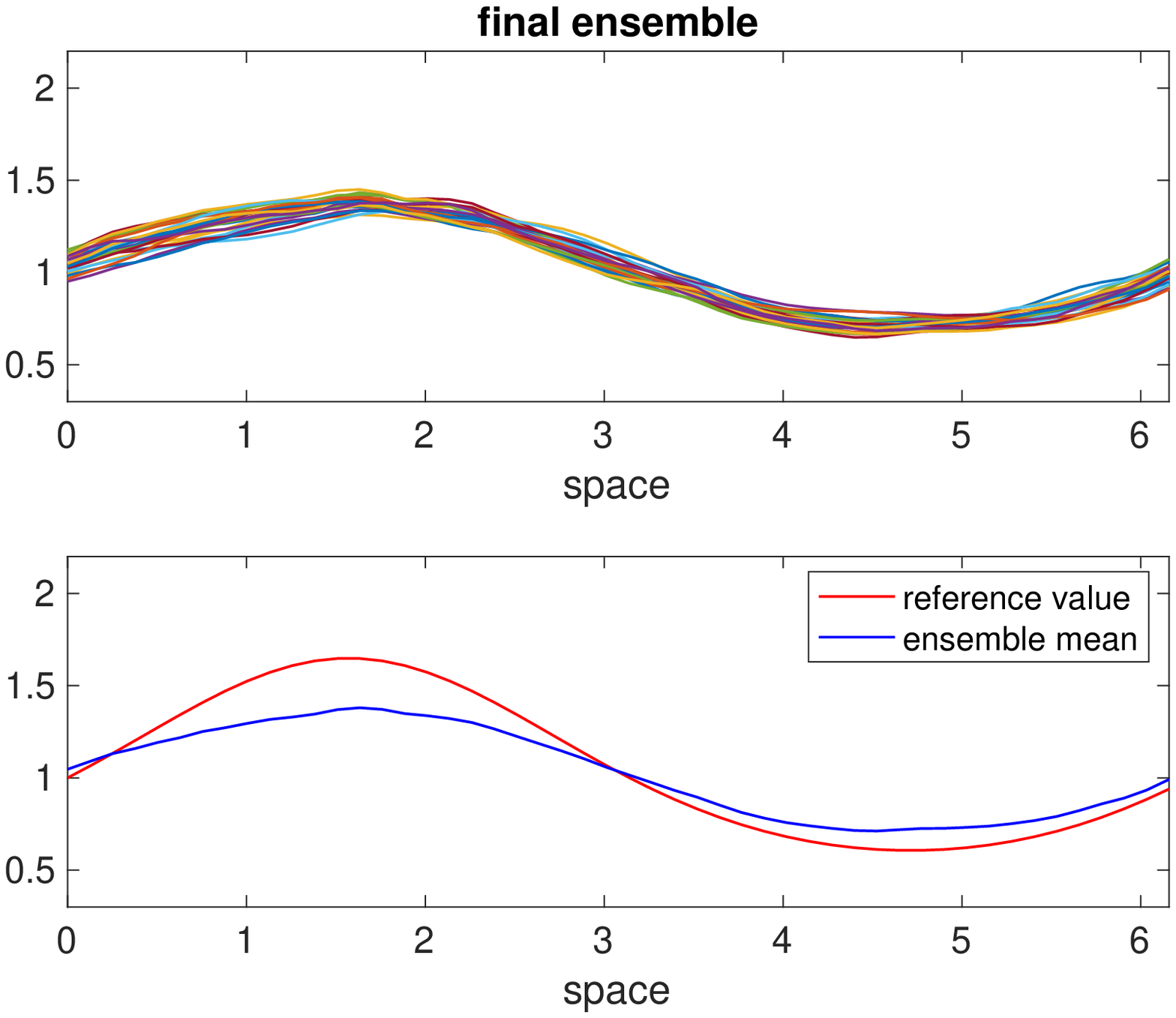} $\qquad$
\includegraphics[width=0.45\textwidth,trim = 0 0 0 0,clip]{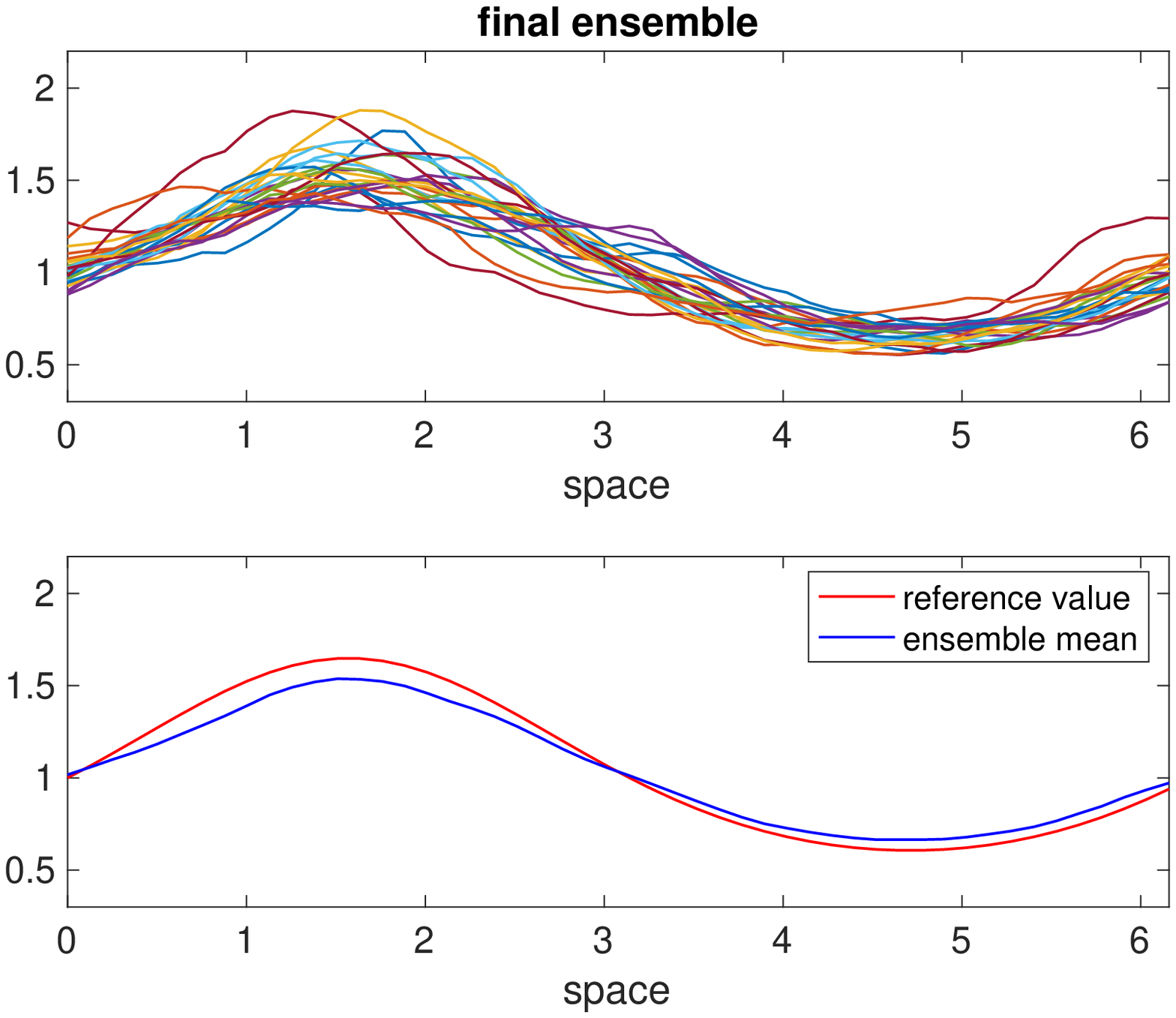}
\end{center}
\caption{Displayed are the initial (top row) and final (bottom row) ensembles  of the permeability fields $a(x) = \exp(u(x))$
for $N=25$. The left column is from the EKS while the right column is from the ALDI method.}
\label{figure1}
\end{figure}

\begin{figure}
\begin{center}
\includegraphics[width=0.45\textwidth,trim = 0 0 0 0,clip]{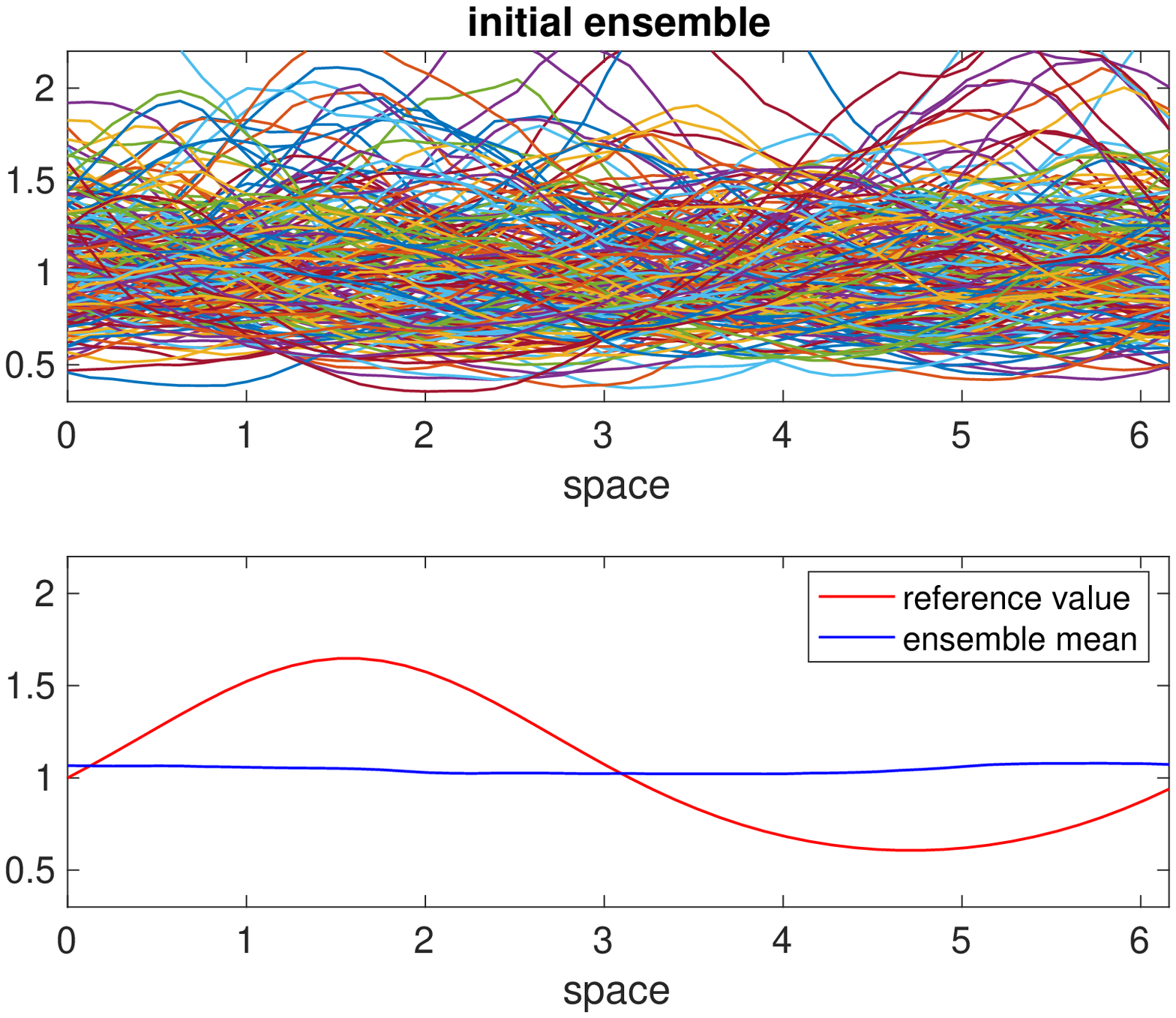} $\qquad$
\includegraphics[width=0.45\textwidth,trim = 0 0 0 0,clip]{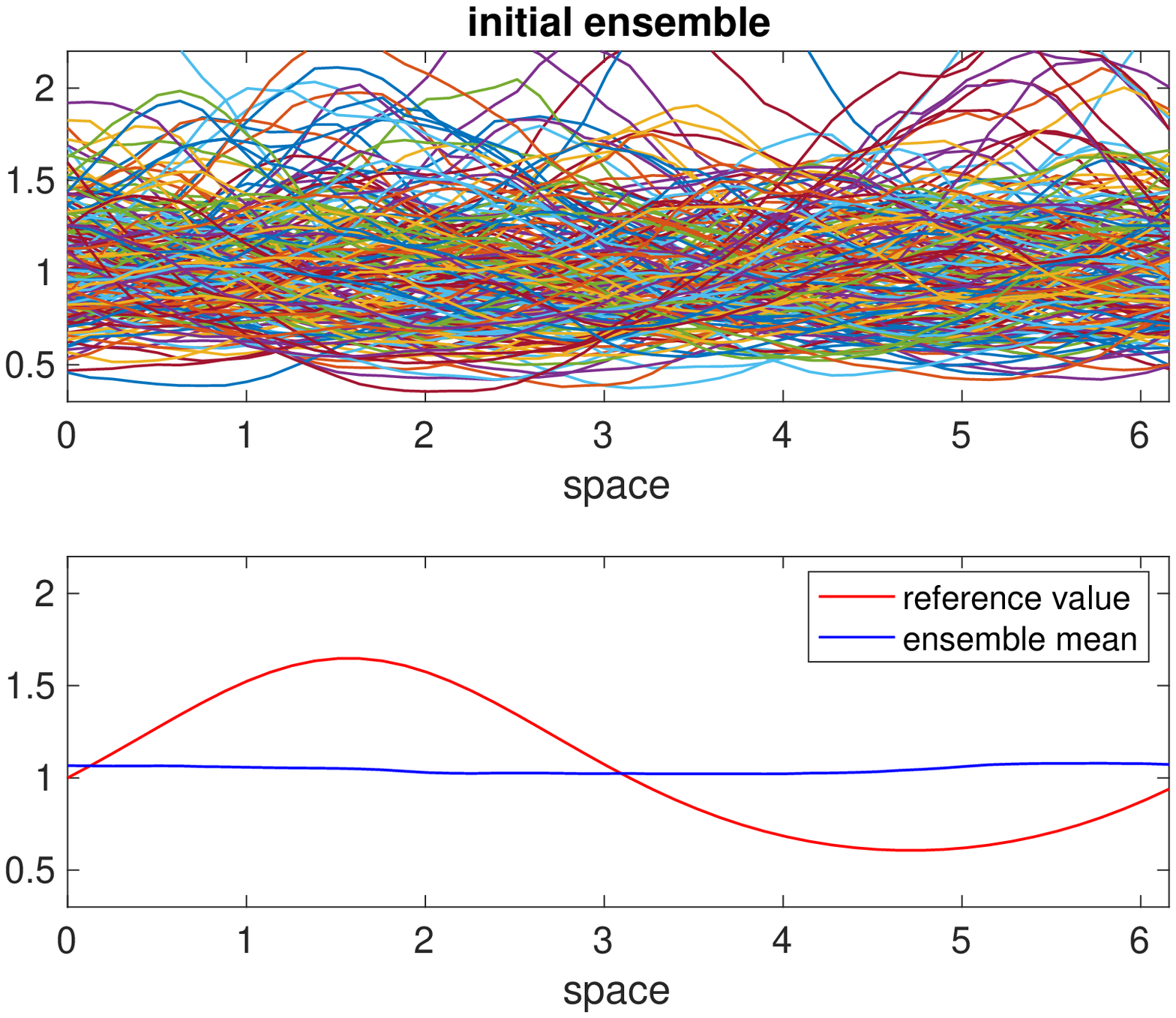} \\ \medskip
\includegraphics[width=0.45\textwidth,trim = 0 0 0 0,clip]{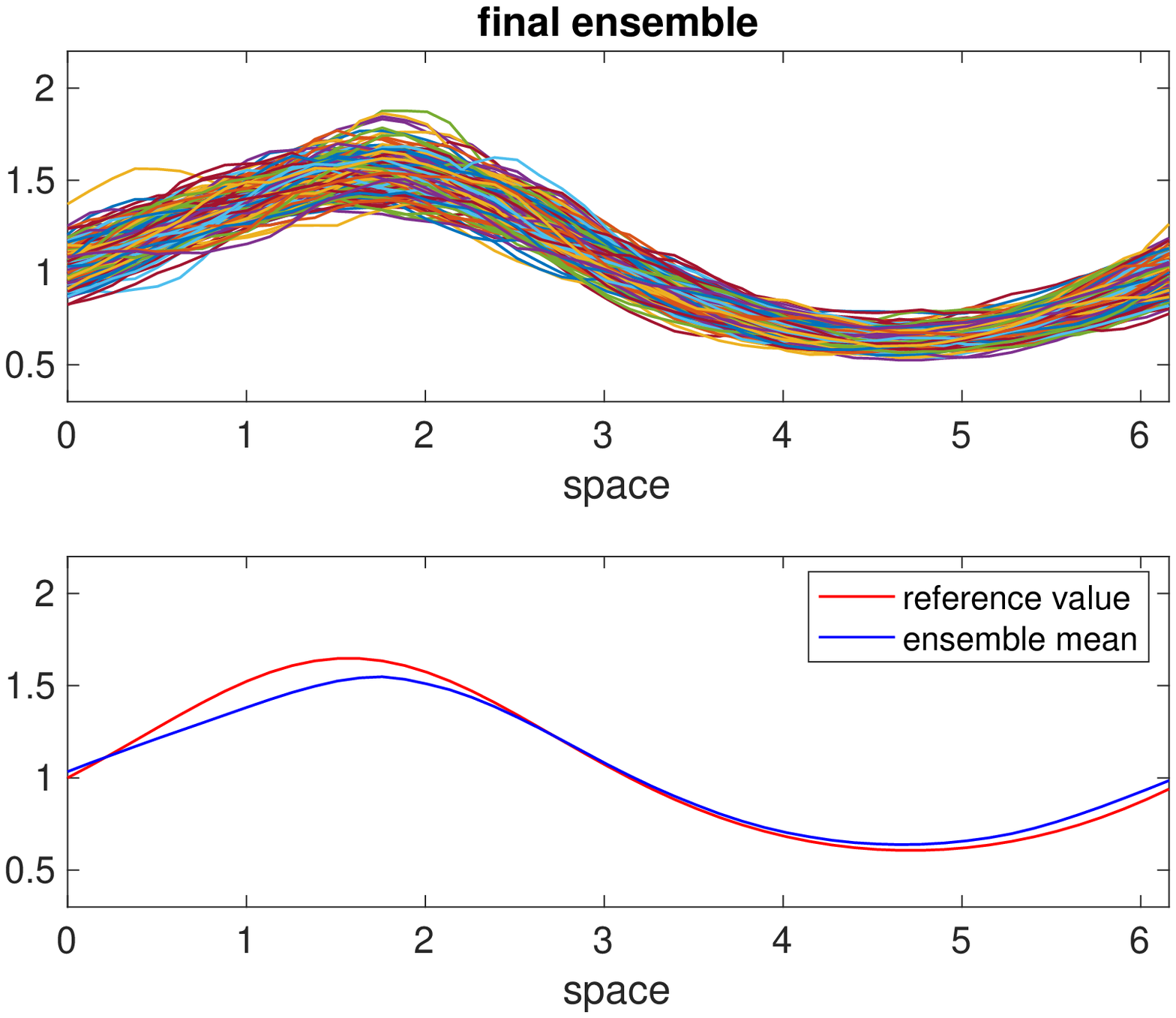} $\qquad$
\includegraphics[width=0.45\textwidth,trim = 0 0 0 0,clip]{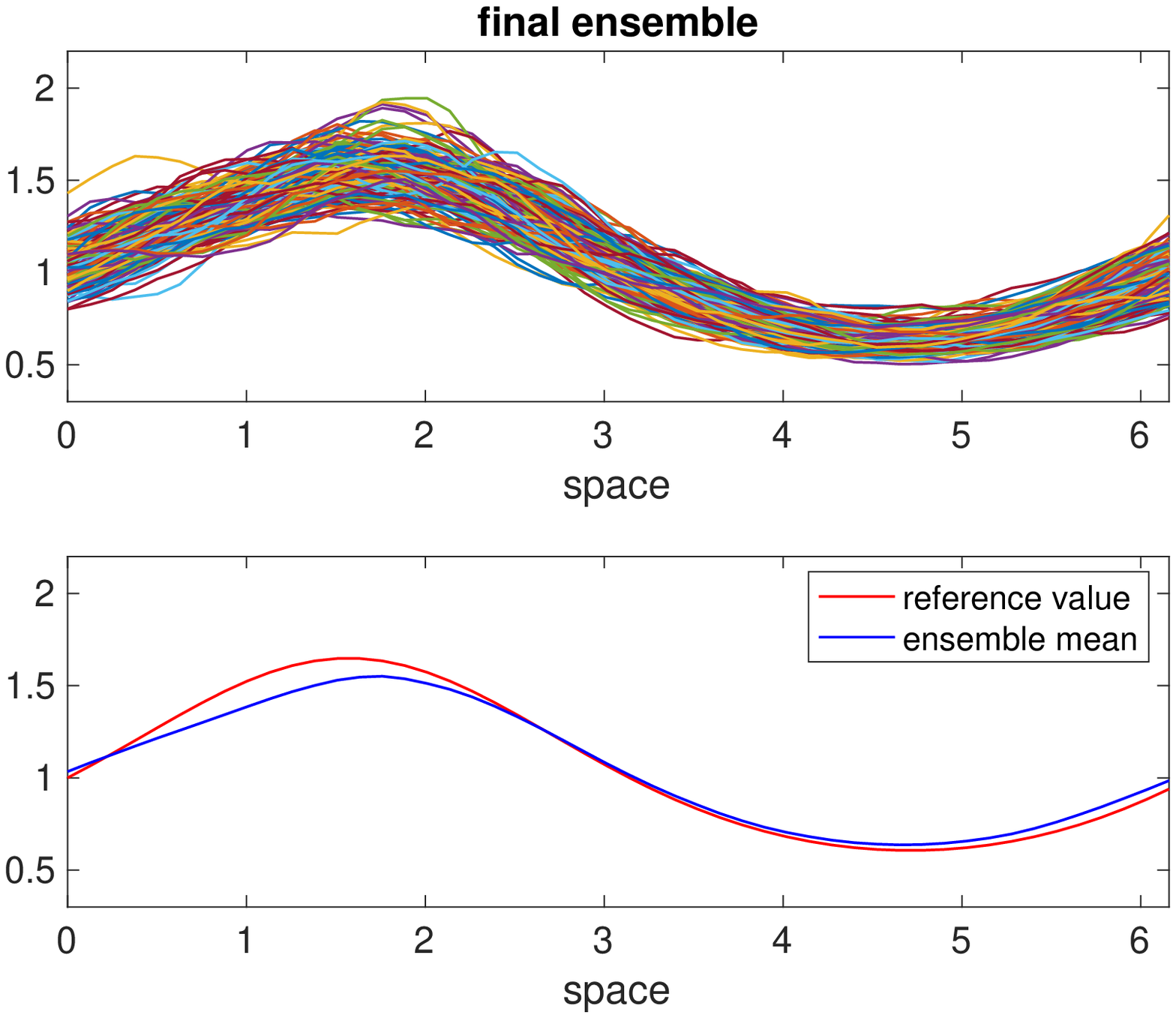}
\end{center}
\caption{As in Figure \ref{figure1}, except for ensemble size $N = 200$.}
\label{figure2}
\end{figure}

We conclude from this simple experiment that the correction term (\ref{eq:correction_term}) is required
for implementations of ALDI whenever the ensemble size is of the order of the dimension of the parameter space
or less. The experiments also confirm that gradient-free implementations can offer a computationally attractive
alternative to gradient-based implementations of ALDI.

%
\section{Conclusions}
%

We have proposed a finite ensemble size implementation of the Kalman--Wasserstein gradient flow formalism put forward in \cite{sr:GIHLS19},
which requires the inclusion of a correction term (\ref{eq:correction_term}) due to the multiplicative nature of the noise in the Langevin equations
(\ref{eq:SDE}) \cite{sr:NR19}. In addition to sampling from the desired target distribution, it has also been demonstrated that the equations of motion
are affine invariant. While ALDI can be used with $N\le D$ ensemble members, effectively leading to a linear subspace sampling method,
it has also been proven that $N>D+1$ and a non-singular initial empirical covariance matrix $\cC(U_0)$ ensure that $|\cC(
U_t)| \not= 0$ for all $t \ge 0$ and that the equations of motion (\ref{eq:SDE}) are ergodic with invariant measure $\pi^{(N)}_\ast$.
Further computational savings can be achieved through the gradient-free implementation (\ref{eq:SDE_DF}) for BIPs as introduced in Example
\ref{ex:BIP}. The effectiveness of gradient-free affine invariant sampling methods has been demonstrated for a Darcy flow inversion problem.
This example has also demonstrated the significance of the correction term for reducing estimation errors both for $N<D$ as well as for $N = \cO(D)$
implementations of the ALDI method (\ref{eq:SDE}).

A numerical issue which has not been studied in this paper is the choice of an efficient time-stepping method for ALDI. In particular,
adaptive and semi-implicit time-stepping methods might be necessary whenever the initial distribution $\pi_0$ is
not close to the target measure $\pi_\ast$. This issue has been studied for
the related continuous-time ensemble Kalman--Bucy filter in \cite{sr:akir11}. We also reemphasise that multi-model target distributions might
require localised empirical covariance matrices in (\ref{eq:SDE}) as first suggested in \cite{sr:LMW18} and further explored in
\cite{sr:RW19}.

While this paper has focused on a theoretical investigation and computational implementation of finite-sample size
interacting Langevin dynamics, we wish to point out that the Kalman--Wasserstein gradient flows proposed in
\cite{sr:cotterreich,sr:GIHLS19} have also become the focus of theoretical studies. We mention in particular
\cite{ding2019ensemble}, which provides a rigorous mean field limit with rates in Wasserstein-2 for the linear case, and \cite{carrillo2019wasserstein}, which studies the decay for the mean field limit in Wasserstein-2 in the linear case using
explicitly the dynamics of the covariance matrix.

\medskip \medskip

\noindent {\bf Acknowledgement.} This research has been partially funded by
Deutsche Forschungsgemeinschaft (DFG, German Science Foundation) - SFB 1294/1 - 318763901 and SFB 1114/2
235221301. AGI is supported by the generosity of Eric and Wendy Schmidt by
recommendation of the Schmidt Futures program, by Earthrise Alliance, by the
Paul G. Allen Family Foundation, and by the National Science Foundation (NSF
grant AGS‐-1835860). We would like to thank Christian B\"ar, Andrew Duncan,
Franca Hoffmann, Andrew Stuart, and Jonathan Weare for valuable discussions
related to the sampling methods proposed in this paper.

%
\appendix
%


\section*{Appendix: Proofs for non-degeneracy and ergodicity}

\begin{proof}[Proof of Proposition \ref{prop:linear FPE}]
The Fokker--Planck equation is given by
\begin{equation}
\partial_t \pi^{(N)}_t = \mathcal{L}^\dagger \pi^{(N)}_t,
\end{equation}
where $\mathcal{L}$ denotes the infinitesimal generator of \eqref{eq:SDE} and $\mathcal{L}^\dagger$ refers to its adjoint in $L^2(\mathbb{R}^{D \times N})$,
given by
\begin{subequations}
\begin{align}
\left( \cL^\dagger \pi^{(N)} \right) (U) &= \sum_{i=1}^N \nabla_{u^{(i)}} \cdot \left(\pi^{(N)}(U) \left\{\cC(U)
\nabla_{u^{(i)}} \Phi (u^{(i)}) -
\frac{D+1}{N}(u^{(i)}-m(U)\right\} \right) \\
& \qquad + \,  \sum_{i=1}^N \nabla_{u^{(i)}} \cdot \left\{\pi^{(N)}(U) \,\nabla_{u^{(i)}} \cdot \cC(U)
+ \cC(U) \nabla_{u^{(i)}} \pi^{(N)}(U) \right\},
\end{align}
\end{subequations}
see \cite[Chapter 4]{sr:P14} and \cite{sr:LMW18}. Here the divergence of the matrix-valued $\cC(U)$ is component-wise
given by (in terms of the notation introduced in Section \ref{sec:gradient_flow_structure})
\begin{equation}
\left\{ (\nabla_{u^{(i)}} \cdot \cC(U)) \right\}_{k} = \sum_{\gamma =1}^D \frac{\partial}{\partial U^{(\gamma,i)}}
\left\{\cC(U)\right\}_{k\gamma}, \qquad
k = 1,\ldots,D.
\end{equation}
An explicit calculation  \cite{sr:NR19} leads to (\ref{eq:div calculations}) and the Fokker--Planck operator $\cL^\dagger$ reduces to
\begin{subequations} \label{eq:FPO}
\begin{align}
\cL^\dagger \pi^{(N)} &= \sum_{i=1}^N \nabla_{u^{(i)}} \cdot \left(\pi^{(N)} \,\cC \left\{
\nabla_{u^{(i)}} \Phi + \nabla_{u^{(i)}} \log \pi^{(N)} \right\} \right) \\ &= \sum_{i=1}^N
\nabla_{u^{(i)}} \cdot \left( \pi^{(N)} \,\cC \,\nabla_{u^{(i)}} \log \frac{\pi^{(N)}}{\pi^{(N)}_\ast} \right)
\end{align}
\end{subequations}
from which the desired result follows since
\begin{equation}
\frac{\delta {\rm KL}(\pi^{(N)}| \pi^{(N)}_\ast)}{\delta \pi^{(N)}} = \log \frac{\pi^{(N)}}{\pi^{(N)_\ast}}.
\end{equation}
\end{proof}

\noindent
For the proof of Proposition \ref{prop:nondegeneracy} we recall the definition (\ref{eq:invertible set}) of the set
$E \subset \mathbb{R}^{D\times N}$.
We will use the potential $\cV$ defined in \eqref{eq:potential} as a Lyapunov function. The key calculation is summarised in the following lemma:
\begin{lemma}
	\label{lem:key lemma}
	There exists a constant $\gamma > 0$  such that
	\begin{equation}
	\label{eq:Lyapunov bound}
	(\mathcal{L} \cV)(U) \le \gamma  \cV(U),  \quad U \in E,
	\end{equation}
	where $\mathcal{L}$ is the generator of \eqref{eq:SDE}, that is, the $L^2(\mathbb{R}^{D \times N})$-adjoint of $\cL^\dagger$ as defined in (\ref{eq:FPO}).
\end{lemma}
\begin{proof}
	It follows from (\ref{eq:FPO}) that the generator of \eqref{eq:SDE} takes the form
	\begin{equation}
	\label{eq:generator}
	(\mathcal{L} \cV)(U) = -\sum_{i=1}^N \nabla_{u^{(i)}} \Phi(u^{(i)}) \cdot \cC(U) \nabla_{u^{(i)}}\cV(U) + \sum_{i=1}^N \nabla_{u^{(i)}} \cdot \left( \cC(U) \nabla_{u^{(i)}} \cV(U) \right).
	\end{equation}
	For convenience, let us introduce the notation
	\begin{equation}
	\cV_{\cC}(U) = - \frac{D+1}{2}\log | \cC(U)|, \qquad \cV_\Phi(U) =  \sum_{i=1}^N \Phi(u^{(i)}).
	\end{equation}
	Since $\mathcal{L}$ vanishes on constants, \eqref{eq:Lyapunov bound} is  equivalent to  $\mathcal{L} \cV \le \gamma \cV + \tilde{C}$ for some constant $\tilde{C}$. Here and in the following, $\tilde{C}$ denotes a generic constant that can change from line to line. Furthermore, by the growth condition on $\Phi$ there exists a constant $\tilde{C}$ such that $-2\cV_\cC \le \cV_\Phi + \tilde{C}$. Therefore, it is sufficient to show the bound $\mathcal{L}\cV \le \tilde{C} (1 + \cV_\Phi)$. In the remainder of the proof, we achieve the latter bound  term-wise for the contributions in \eqref{eq:generator}.

	First note that
	\begin{equation}
	\nabla_{u^{(i)}} \log | \cC(U)| = \frac{2}{N} \cC^{-1}(U)(u^{(i)} - m(U)), \qquad i =1, \ldots, N.
	\end{equation}
	Indeed, again following the notation introduced in Section \ref{sec:gradient_flow_structure}, we have that
	\begin{subequations}
		\label{eq:log det calculation}
		\begin{align}
		\frac{\partial}{\partial U^{(\gamma,i)}} \log | \cC(U)| & = \left( \frac{1}{| \cC|} \sum_{\alpha,\beta  = 1}^D \frac{\partial | \cC| }{\partial \cC^{\alpha \beta}} \frac{\partial \cC^{\alpha \beta}}{\partial U^{(\gamma,i)}} \right)(U) \\
		& = \sum_{\alpha,\beta  = 1}^D (\cC^{-1})_{\alpha \beta} \left(\frac{1}{N} \delta_{\alpha \gamma} (u^{(i)} - m(U))_\beta + \frac{1}{N} \delta_{\beta \gamma} (u^{(i)} - m(U))_\alpha\right)
		\\
		& = \frac{2}{N}\left(\cC^{-1} (u^{(i)} - m(U))\right)_\gamma,
		\end{align}
	\end{subequations}
	using Jacobi's formula for determinants in the second line.
	For the first term in \eqref{eq:generator} we thus obtain
	\begin{subequations}
		\label{eq:estimate1}
		\begin{align}
		-& \sum_{i=1}^N \nabla_{u^{(i)}} \Phi(u^{(i)}) \cdot \cC(U) \nabla_{u^{(i)}}\cV(U)\\
		&= - \frac{D+1}{N}\sum_{i=1}^N \nabla_{u^{(i)}} \Phi(u^{(i)}) \cdot (u^{(i)} - m(U))
		\underbrace{- \sum_{i=1}^N \nabla_{u^{(i)}} \Phi(u^{(i)}) \cdot \cC(U) \nabla_{u^{(i)}} \Phi(u^{(i)})}_{\le 0} \\
		& \le\tilde{C}\left(1 + \cV_\Phi(U)\right).
		\end{align}
	\end{subequations}
	To bound the second term in \eqref{eq:generator}, we first notice the estimate
	\begin{equation}
	\label{eq:estimate2}
	\sum_{i=1}^N \nabla_{u^{(i)}} \cdot \left( \cC(U) \nabla_{u^{(i)}} \Phi(u^{(i)}) \right) \le \tilde{C}\left(1 + \cV(U)\right),
	\end{equation}
	again easily obtained from Assumption \ref{ass:V}. The other contribution is
	\begin{subequations}
	\begin{align}
	\sum_{i=1}^N \nabla_{u^{(i)}} \cdot \left(\cC(U) \nabla_{u^{(i)}} \cV_\cC(U)\right) &= -\frac{D+ 1}{N} \sum_{i=1}^N \nabla_{u^{(i)}} \cdot \left(u^{(i)} - m(U)\right) \\
	&= - \frac{(D+1)D}{2N} \left(N-1\right).
	\end{align}
	\end{subequations}
	Since the result is a constant, we clearly have the required estimate of the form $\mathcal{L}\cV \le \tilde{C} (1 + \cV_\Phi)$. In conjunction with \eqref{eq:estimate1} and \eqref{eq:estimate2} the claim follows.
\end{proof}

\noindent
Proposition \ref{prop:nondegeneracy} now essentially follows from adapting \cite[Theorem 2.1]{meyn1993stability3}. Textbook accounts of similar arguments can be found in \cite[Chapter 5]{khasminskii2011stochastic} and \cite[Chapter 2]{eberle2009markov}. For the convenience of the reader we provide a self-contained proof:

\begin{proof}[Proof of Proposition \ref{prop:nondegeneracy}]
	The potential $\cV$ is bounded from below by the growth condition on $\Phi$ (see Assumption \ref{ass:V}). We can therefore choose a constant $c_\cV$ such that $\cV_+ := \cV + c_\cV$ is nonnegative.
	Since $\cC(U_0)$ is assumed to be nondegenerate, there exists $k_0 \in \mathbb{N}$ such that $\cV_+(U_0) < k_0$.
	For $k \ge k_0$, let us define the sets
	\begin{equation}
	E_k = \left\{U \in E:  \quad \cV_+(U) < k \right\}
	\end{equation}
	and the stopping times
	\begin{equation}
	\tau_k =\inf \{  t \ge 0\ :  U_t \notin E_k \}  = \inf \left\{t \ge 0: \cV_+(U_t) = k \right\}.
	\end{equation}
	The stopping times $\tau_k$ are increasing in $k$, and so the limit $$\lim_{k \rightarrow \infty} \tau_k =: \xi$$ exists in $[0,+\infty]$. To prove the claim, it is sufficient to show that $\mathbb{P}[\xi = +\infty] = 1$.
	We now define
	\begin{equation}
	g(U,t) := e^{-\gamma t} \cV_+(U), \quad (U,t) \in E \times [0,\infty),
	\end{equation}
	where $\gamma$ is the constant obtained in Lemma \ref{lem:key lemma}. By using It{\^o}'s formula, optional stopping, and the bound \eqref{eq:Lyapunov bound} we see that
	\begin{subequations}
		\label{eq:supermartingale}
		\begin{align}
		\mathbb{E}  \left[ g(U_{t \wedge \tau_k}, t \wedge \tau_k) \right] & = g(U_0,0)  + \mathbb{E} \left[ \int_0^{t \wedge \tau_k} e^{-\gamma s} \left( -\gamma \cV_+(U_s) + \mathcal{L} \cV_+(U_s)\right)\mathrm{d}s\right]
		\\
		 & \le g(U_0,0) = \cV_+(U_0),
		\end{align}
	\end{subequations}
	for any $t \ge 0$ and $k \ge k_0$. On the other hand,
	\begin{subequations}
		\label{eq:chebyshev}
		\begin{align}
		\mathbb{E}  \left[ g(U_{t \wedge \tau_k}, t \wedge \tau_k) \right] &\ge e^{-\gamma t} \mathbb{E} \left[\cV_+ (U_{t \wedge \tau_k})\right]
		\\
		& \ge e^{-\gamma t} \left(\mathbb{E}\left[\mathbf{1}_{t < \tau_k} \cV_+(U_t)\right] + \mathbb{P}\left[\tau_k \le t \right] \cdot k\right) \ge e^{-\gamma t} \left( \mathbb{P}\left[\tau_k \le t \right] \cdot k\right),
		\end{align}
	\end{subequations}
	where the last estimate uses the fact that $\cV_+ \ge 0$. Combining \eqref{eq:supermartingale} and \eqref{eq:chebyshev}, we see that
	\begin{equation}
	e^{\gamma t} \cV_+(U_0) \ge \mathbb{P}\left[\tau_k \le t \right] \cdot k,
	\end{equation}
	for every $t \ge 0$ and $k \ge k_0$. It follows immediately that
	\begin{equation}
	\lim_{k \rightarrow \infty} \mathbb{P}\left[\tau_k \le t \right] = 0,
	\end{equation} and further
	\begin{equation}
	\label{eq:limit stopping}
	\mathbb{P}\left[\xi \le t \right] = 0
	\end{equation}
	by monotone convergence. Since \eqref{eq:limit stopping} holds for all $t \ge 0$, we conclude that $\mathbb{P}[\xi = \infty] =1$, as required.
\end{proof}

\noindent
For the proof of Proposition \ref{prop:ergodicity} we will need the following lemma:

\begin{lemma}
	\label{lem:path connected}
	Let $N \ge D+2$. Then $E$ is path-connected.
\end{lemma}
\begin{proof}
We begin by fixing some additional notation. For $j \in \{1, \ldots, N\}$, define the `leave-one-out' versions of the empirical mean and covariance,
\begin{equation}
m^{-j} (U) = \frac{1}{N-1} \sum_{\substack{i=1 \\ i \neq j}}^N u^{(i)}, \qquad
\cC^{-j} (U) = \frac{1}{N-1} \sum_{\substack{i=1 \\ i \neq j}}^N (u^{(i)} - m^{-j}(U)) (u^{(i)} - m^{-j}(U))^{\rm T}.
\end{equation}
Notice the update formula
\begin{equation}
\label{eq:update}
\cC(U) = \frac{N-1}{N} \cC^{-j}(U) + \frac{N-1}{N^2} (u^{(j)} - m^{-j}(U))(u^{(j)} - m^{-j}(U))^{\rm T},
\end{equation}
holding for any $j \in \{1, \ldots N\}$. Consider now the set
\begin{equation}
\widetilde{E} := \left\{U\in \mathbb{R}^{D \times N}: \quad \cC^{-j} \,\,\text{is invertible for all } j \in \{1, \ldots , N\}   \right\}.
\end{equation}
We see that $\widetilde{E} \subset E$ since the second term on the right-hand side of \eqref{eq:update} is positive semidefinite. Importantly, the condition $N \ge D + 2$ ensures that $\widetilde{E}$ is nonempty.

Observe that $\widetilde{E}$ has the representation
\begin{equation}
\widetilde{E} = \left\{ U \in \mathbb{R}^{D \times N}: \quad \prod_{j=1}^N \vert \cC^{-j}(U) \vert > 0  \right\},
\end{equation}
immediately implying that $\widetilde{E}$ is open. We now show that $\widetilde{E}$ is dense in $\mathbb{R}^{D \times N}$.
To this end, let $X \in \mathbb{R}^{D \times N}$ and $Y \in \widetilde{E}$. It is sufficient to prove that for every $\varepsilon > 0$ there exists $t \in (0,\varepsilon)$ such that $(1-t)X + tY \in \widetilde{E}$. For this, define $P:\mathbb{R} \rightarrow \mathbb{R}$ by
\begin{equation}
P(t) = \prod_{j=1}^N \vert \cC^{-j}((1-t)X + tY) \vert,
\end{equation}
which is clearly a polynomial. Since $Y \in \widetilde{E}$ we have that $P(1) > 0$, and so $P$ has only finitely many zeroes. This proves that indeed for all $\varepsilon > 0$ there exists $t \in (0,\varepsilon)$ such that $P(t) > 0$, and hence $(1-t)X + tY \in \widetilde{E}$.

We now show how to construct a continuous path between arbitrary $X,Y \in E$. By density of $\widetilde{E}$, it is enough to find a path between $\widetilde{X} \in \widetilde{E}$ and $\widetilde{Y} \in \widetilde{E}$ lying in connected neighbourhoods of $X$ and $Y$ respectively. Since $\widetilde{E}$ is open, there exist open neighbourhoods $U_{\widetilde{X}},U_{\widetilde{Y}} \subset \widetilde{E}$. It is then sufficient to find points in these neighbourhoods that can be connected by a continuous path. Denoting $\widetilde{X} = (x^{(1)}, \ldots, x^{(N)})$ and $\widetilde{Y} = (y^{(1)}, \ldots, y^{(N)})$, we can choose a continuous path $\gamma^{(1)}:[0,1] \rightarrow \mathbb{R}^D$ with $\gamma^{(1)}(0) = x^{(1)}$ and $\gamma^{(1)}(1) = y^{(1)}$, and set $\overline{\gamma}^{(1)}(t) = (\gamma^{1}(t), x^{(2)}, \ldots, x^{(N)})$. By \eqref{eq:update}, it is clear that $\overline{\gamma}^{(1)}(t) \in E$ for all $t \in [0,1]$. By density of $\widetilde{E}$ we can perturb $\overline{\gamma}^{(1)}(1)$ in order to ensure that   $\overline{\gamma}^{(1)}(1) \in \widetilde{E}$. We can now proceed iteratively to move the remaining particles using paths $\overline{\gamma}^{(2)},\ldots \overline{\gamma}^{(N)}$ and concatenate them, yielding the required total path. Note that the perturbation of the endpoints of $\overline{\gamma}^{(i)}$ can be chosen arbitrarily small in order to ensure that the final point $\overline{\gamma}^{(N)}(1)$ belongs to $U_{\widetilde{Y}}$.
\end{proof}

\begin{proof}[Proof of Proposition \ref{prop:ergodicity}]
Since the diffusion matrix $\Gamma(U) \Gamma(U)^{\rm T}$
is strictly positive definite on $E$, $E$ is path-connected by Lemma \ref{lem:path connected}, and the process
$(U_t)_{t \ge 0}$ admits an invariant measure with strictly positive Lebesgue-density by Corollary \ref{cor:invariance}, the process is positively recurrent and irreducible by the result in \cite{kliemann1987recurrence}.	We also refer to \cite[Section 2.2.2.1]{stoltz2010free}. The convergence in total variation distance then follows from \cite[Theorem 6.1]{meyn1993stability2}.
\end{proof}

%
%
\bibliographystyle{plain}
\bibliography{bib_paper}
%
%

\end{document}